\documentclass[11pt]{amsart}
\usepackage{amssymb,amsmath,amsthm,andy,fullpage}

\usepackage{fancyhdr}
\newcommand{\Toz}{T^{1,0}}
\newcommand{\Tzo}{T^{0,1}}

\newcommand{\Lb}{\bar L}
\newcommand{\Lba}{ {\bar L}^* }
\newcommand{\Lbap}{ {\bar L}^{*,+} }
\newcommand{\Lbam}{ {\bar L}^{*,-} }
\newcommand{\Lbal}{ {\bar L}^{*,\lam} }

\newcommand{\dbarbsp}{\dbarb^{*,+}}
\newcommand{\dbarbsm}{\dbarb^{*,-}}
\newcommand{\dbarbspm}{\dbars_{b,\pm}}
\newcommand{\dbarbsl}{\dbarb^{*,\lam}}
\newcommand{\dbarbst}{\dbar^{*}_{b,t}}

\newcommand{\Boxbt}{\Box_{b,t}}

\newcommand{\Qbpm}{Q_{b,\pm}}
\newcommand{\Qbpmp}{\Qbpm(\vp,\vp)}
\newcommand{\Qbp}{Q_{b,+}}

\newcommand{\Qbm}{Q_{b,-}}

\newcommand{\Qbo}{Q_{b,0}}

\newcommand{\Qblp}{Q_{b,\lam}(\vp,\vp)}
\newcommand{\Qbt}{Q_{b,t}}

\newcommand{\ob}{\,\bar\omega}
\newcommand{\omb}{\bar\omega}

\newcommand{\Cp}{\mathcal{C}^+}
\newcommand{\Cm}{\mathcal{C}^-}
\newcommand{\Co}{\mathcal{C}^0}

\newcommand{\Cpn}{\mathcal{C}^+_{\nu}}
\newcommand{\Cmn}{\mathcal{C}^-_{\nu}}
\newcommand{\Con}{\mathcal{C}^0_{\nu}}

\newcommand{\Com}{\mathcal{C}^0_{\mu}}

\newcommand{\tCon}{\tilde{\mathcal{C}}^0_\nu}

\newcommand{\tCpm}{ {\tilde{\mathcal{C}}}^+_{\mu} }
\newcommand{\tCmm}{\tilde{\mathcal{C}}^-_\mu}

\newcommand{\psp}{\psi^+}
\newcommand{\psm}{\psi^-}
\newcommand{\pso}{\psi^0}

\newcommand{\pspl}{\psi^+_{t}}
\newcommand{\psml}{\psi^-_{t}}
\newcommand{\psol}{\psi^0_{t}}

\newcommand{\tpsplm}{{\tilde\psi}^+_{\mu,t}}
\newcommand{\tpsmlm}{{\tilde\psi}^-_{\mu,t}}

\newcommand{\Pso}{\Psi^0}
\newcommand{\Pspl}{\Psi^+_{t}}
\newcommand{\Psml}{\Psi^-_{t}}
\newcommand{\Psol}{\Psi^0_{t}}

\newcommand{\tPsol}{{\tilde\Psi}^0_{t}}

\newcommand{\Pspla}{ (\Psi^+_{t})^* }
\newcommand{\Psmla}{ (\Psi^-_{t})^* }
\newcommand{\Psola}{ (\Psi^0_{t})^* }

\newcommand{\Pspln}{\Psi^+_{\nu,t}}
\newcommand{\Psmln}{\Psi^-_{\nu,t}}
\newcommand{\Psoln}{\Psi^0_{\nu,t}}

\newcommand{\tPsoln}{{\tilde\Psi}^0_{\nu,t}}

\newcommand{\Psplan}{ (\Psi^+_{\nu,t})^* }
\newcommand{\Psmlan}{ (\Psi^-_{\nu,t})^* }
\newcommand{\Psolan}{ (\Psi^0_{\nu,t})^* }

\newcommand{\Psplm}{\Psi^+_{\mu,t}}
\newcommand{\Psmlm}{\Psi^-_{\mu,t}}
\newcommand{\Psolm}{\Psi^0_{\mu,t}}
\newcommand{\tPsplm}{{\tilde\Psi}^+_{\mu,t}}
\newcommand{\tPsmlm}{{\tilde\Psi}^-_{\mu,t}}

\newcommand{\Psplam}{ (\Psi^+_{\mu,t})^* }

\newcommand{\Jt}{ {\mathcal{J}}_{\vartheta} }
\newcommand{\tJt}{ {}^t\!{\mathcal{J}}_{\vartheta} }

\newcommand{\lp}{{\lambda^{+}}}
\newcommand{\lm}{{\lambda^{-}}}

\newcommand{\zn}{\zeta_\nu}
\newcommand{\tzn}{ {\tilde\zeta_{\nu}} }
\newcommand{\zm}{\zeta_\mu}
\newcommand{\tzm}{ {\tilde\zeta_{\mu}} }

\newcommand{\tz}{{\tilde\zeta}}

\newcommand{\normlplm}{\norm_{\lp,\lm}}
\newcommand{\normpm}{\norm_{\pm}}

\newcommand{\sumn}{\sum_{\nu}}
\newcommand{\summ}{\sum_{\mu}}

\newcommand{\vpn}{\vp^\nu}

\newcommand{\la}{\langle}
\newcommand{\ra}{\rangle}
\newcommand{\ralplm}{\rangle_{\lp,\lm}}
\newcommand{\rapm}{\rangle_{\pm}}

\newcommand{\Hpm}{\mathcal{H}_{\pm}}
\newcommand{\Hpmp}{{}^\perp\Hpm}

\newcommand{\Boxbpm}{\Box_{b,\pm}}

\newcommand{\Gqdt}{G_{q,t}^\delta}

\renewcommand{\H}{\mathcal{H}}

\newcommand{\bfem}[1]{\textbf{\emph{#1}}}

\newcommand{\pnorm}[1]{\left\Vert#1\right\Vert}

\newtheorem{thm}{Theorem}[section]
\newtheorem{prop}[thm]{Proposition}
\newtheorem{lem}[thm]{Lemma}

\theoremstyle{definition}
\newtheorem{defn}[thm]{Definition}

\theoremstyle{remark}
\newtheorem{rem}[thm]{Remark}

\begin{document}

\title [Regularity results for $\dbarb$]{Regularity results for $\dbarb$ on CR-manifolds of hypersurface type}

\author{Phillip S. Harrington and Andrew Raich}

\thanks{The first author is partially supported by NSF grant DMS-1002332
second author is partially supported by NSF grant DMS-0855822}

\address{Department of Mathematical Sciences, SCEN 327, 1 University of Arkansas, Fayetteville, AR 72701}
\email{psharrin@uark.edu \\ araich@uark.edu}

\maketitle

%
%
\section{Introduction and Results}

In this article, we introduce a class of embedded CR manifolds satisfying a geometric condition that we call weak
$Y(q)$. For such manifolds, we show that $\dbarb$ has closed range on $L^2$ and that the complex Green operator is
continuous on $L^2$. Our methods involves building a weighted norm from a microlocal decomposition. We also prove that
at any Sobolev level there is a weight such that the complex Green operator inverting the weighted Kohn Laplacian is
continuous. Thus, we can solve the $\dbarb$-equation in $C^\infty$.

Let $M^{2n-1}\subset\mathbb{C}^N$ be a $C^\infty$ compact, orientable CR-manifold, $N\geq n$.  We say that $M$ is of
hypersurface type if the CR-dimension of $M$ is $n-1$, so that the complex tangent bundle of $M$ splits into a complex
subbundle of dimension $n-1$, the conjugate of the complex subbundle, and one totally real direction.  When the de Rham
complex on $M$ is restricted to the complex subbundle, we obtain the $\dbarb$ complex.

When $M$ is the boundary of a pseudoconvex domain, closed range for $\dbarb$ was obtained in \cite{Shaw85},
\cite{Kohn86}, and \cite{BoSh86}.  This work was extended to pseudoconvex manifolds of hypersurface type by Nicoara in
\cite{Nic06}.  When the domain is not pseudoconvex, there is a condition $Y(q)$ which is known to imply subelliptic
estimates for the complex Green operator acting on $(0,q)$ forms (see \cite{FoKo72} or \cite{ChSh01} for details on
$Y(q)$).  In this article, we will adapt the microlocal analysis used in \cite{Nic06,Rai10c} to obtain closed range
results for $\dbarb$ on manifolds satisfying weak $Y(q)$.

When $M$ is a CR-manifold of hypersurface type, the tangent space of $M$ can be spanned by $(1,0)$ vector fields $L_1,\ldots,L_{n-1}$, their conjugates, and a totally imaginary vector field $T$ spanning the remaining direction.  If $\dbarbs$ denotes the Hilbert space adjoint of $\dbarb$ with respect to the $L^2$ inner product on $M$, we have a basic identity for $(0,q)$ forms $\phi$ of the form
\[
  \pnorm{\dbarb\phi}^2+\pnorm{\dbarbs\phi}^2=\sum_{J\in\mathcal{I}_q}\sum_{j=1}^{n-1}\pnorm{\bar{L}_j\phi_J}^2+\sum_{I\in\mathcal{I}_{q-1}}\sum_{j,k=1}^{n-1}\Rre (c_{jk}T\phi_{jI},\phi_{kI})+\cdots
\]
where $c_{jk}$ denotes the Levi-form of $M$ in local coordinates (see for example the proof of Theorem 8.3.5 in \cite{ChSh01})
and $\I_q$ is the set of increasing $q$-tuples.  The difficulty in using the basic identity to prove regularity estimates for $\dbarb$ rests in controlling the
$\Rre (c_{jk}T\phi_{jI},\phi_{kI})$ terms. When $M$ satisfies $Y(q)$, integration by parts can be performed on the gradient term in such a way that
\[
\pnorm{\dbarb\phi}^2+\pnorm{\dbarbs\phi}^2\geq C(\sum_{J\in\mathcal{I}_q}\sum_{j=1}^{n-1}\pnorm{\bar{L}_j\phi_J}^2+\sum_{J\in\mathcal{I}_q}\sum_{j=1}^{n-1}\pnorm{L_j\phi_J}^2)+\cdots.
\]
Using H\"ormander's classic result on sums of squares \cite{Hor67}, this can be used to estimate $\pnorm{\phi}_{1/2}$.  On manifolds where the Levi-form degenerates, it may still be possible to choose good local coordinates so that with a suitable integration by parts, there is the estimate
\[
\pnorm{\dbarb\phi}^2+\pnorm{\dbarbs\phi}^2\geq \sum_{J\in\mathcal{I}_q}\sum_{j=m+1}^{n-1}\pnorm{\bar{L}_j\phi_J}^2+\sum_{J\in\mathcal{I}_q}\sum_{j=1}^{m}\pnorm{L_j\phi_J}^2+\cdots.
\]
for some integer $m$.  Unfortunately, since such an estimate no longer bounds all of the $L_j$ and $\bar{L}_j$
derivatives, it is not possible to control $\pnorm{\phi}_{1/2}$.  Hence, a weight function is needed to provide some
positivity in the $L^2$-norm. The key idea in \cite{Nic06,Rai10c} is to microlocalize and decompose a form $\phi$ into
pieces whose Fourier transform is supported on specific regions. The authors then build a weighted norm based on the
decomposition.  In this weighted $L^2$-space, the  $c_{jk}T$ terms are under control and a basic estimate holds. If the
weight function is $t|z|^2$, then Nicoara proves that $\dbarb$ has closed range in $L^2$ and in $H^s$, and if the
weight function is obtained from property $(P_q)$, then Raich shows that the complex Green operator is compact on
$H^s(M)$ for all $s\geq 0$.

It is already known through an integration by parts argument (see the work of Ahn, Baracco and Zampieri \cite{AhBaZa06}
or Zampieri \cite{Zam08}) that  local regularity estimates hold on a class of domains where the Levi-form has
degeneracies and mixed signature (known as $q$-pseudoconvex domains).  Our method is to apply microlocal analysis  to
the integration by parts argument used in the $q$-pseudoconvex case to obtain a more general sufficient condition for
(global) $L^2$ and Sobolev space estimates.

Our main results are the following.

%
%
\begin{thm}\label{thm:main theorem for unweighted}
Let $M^{2n-1}$ be a $C^\infty$ compact, orientable weakly $Y(q)$ CR-manifold embedded in $\C^N$, $N\geq n$ and $1\leq q \leq n-2$. Then the following hold:
\begin{enumerate}\renewcommand{\labelenumi}{(\roman{enumi})}
\item The operators $\dbarb: L^2_{0,q}(M)\to L^2_{0,q+1}(M)$ and $\dbarb: L^2_{0,q-1}(M)\to L^2_{0,q}(M)$ have closed range;
\item The operators $\dbarbs: L^2_{0,q+1}(M)\to L^2_{0,q}(M)$ and $\dbarbs: L^2_{0,q}(M)\to L^2_{0,q-1}(M)$ have closed range;
\item The Kohn Laplacian defined by $\Box_b = \dbarb\dbarbs + \dbarbs\dbarb$ has closed range on $L^2_{0,q}(M)$;
\item The complex Green operator $G_q$  is continuous on $L^2_{0,q}(M)$;
\item The canonical solution operators
for $\dbarb$, $\dbarbs G_q:L^2_{0,q}(M)\to L^2_{0,q-1}(M)$ and\\ $G_q\dbarbst : L^2_{0,q+1}(M)\to L^2_{0,q}(M)$,
are continuous;
\item The canonical solution operators for $\dbarbs$, $\dbarb G_q:L^2_{0,q}(M)\to L^2_{0,q+1}(M)$
and\\ $G_q\dbarb : L^2_{0,q-1}(M)\to L^2_{0,q}(M)$, are continuous;
\item The space of harmonic forms $\H^q(M)$, defined to be the $(0,q)$-forms annihilated by $\dbarb$ and $\dbarbs$ is finite dimensional;
\item If $\tilde{q}=q$ or $q+1$ and $\alpha\in L^2_{0,\tilde{q}}(M)$ so that  $\dbarb \alpha =0$, then there exists $u\in L^2_{0,\tilde{q}-1}(M)$ so that
\[
\dbarb u = \alpha;
\]
\item The Szeg\"o projections $S_q = I - \dbarbs \dbarb G_q$ and $S_{q-1}    = I - \dbarbs G_q \dbarb$ are continuous on
$L^2_{0,q}(M)$ and $L^2_{0,q-1}(M)$, respectively.
\end{enumerate}
\end{thm}

These results will be obtained by studying a family of weighted operators with respect to a norm $\norm\phi\norm_t$ defined in terms of the weights $e^{t|z|^2}$
and  $e^{-t|z|^2}$ and the microlocal decomposition of $\phi$.  For such operators, we will also be able to obtain Sobolev space estimates, as follows:

%
%
\begin{thm}\label{thm:main theorem for weighted spaces}
Let $M^{2n-1}$ be a $C^\infty$ compact, orientable weakly $Y(q)$ CR-manifold embedded in $\C^N$, $N\geq n$.
For $s\geq 0$ there exists $T_s\geq 0$
so that the following hold:
\begin{enumerate}\renewcommand{\labelenumi}{(\roman{enumi})}
\item The operators $\dbarb: L^2_{0,q}(M)\to L^2_{0,q+1}(M)$ and $\dbarb: L^2_{0,q-1}(M)\to L^2_{0,q}(M)$ have closed range
with respect to $\norm\cdot\norm_t$. Additionally, for any $s>0$ if $t\geq T_s$, then
$\dbarb: H^s_{0,q}(M)\to H^s_{0,q+1}(M)$ and $\dbarb: H^s_{0,q-1}(M)\to H^s_{0,q}(M)$ have closed range;
\item  The operators $\dbarbst: L^2_{0,q+1}(M)\to L^2_{0,q}(M)$ and $\dbarbst: L^2_{0,q}(M)\to L^2_{0,q-1}(M)$ have closed range
with respect to $\norm\cdot\norm_t$. Additionally, if $t\geq T_s$, then
$\dbarbst: H^s_{0,q+1}(M)\to H^s_{0,q}(M)$ and $\dbarbst: H^s_{0,q}(M)\to H^s_{0,q-1}(M)$ have closed range;
\item The Kohn Laplacian defined by $\Boxbt = \dbarb\dbarbst + \dbarbst\dbarb$ has closed range on $L^2_{0,q}(M)$
(with respect to $\norm\cdot\norm_t$) and also on
$H^s_{0,q}(M)$ if $t\geq T_s$;
\item The space of harmonic forms $\H^q_t(M)$, defined to be the $(0,q)$-forms annihilated by $\dbarb$ and $\dbarbst$ is finite dimensional;
\item The complex Green operator $G_{q,t}$ is continuous on $L^2_{0,q}(M)$
(with respect to $\norm\cdot\norm_t$) and also on
$H^s_{0,q}(M)$ if $t\geq T_s$;
\item The canonical solution operators
for $\dbarb$, $\dbarbst G_{q,t}:L^2_{0,q}(M)\to L^2_{0,q-1}(M)$ and $G_{q,t}\dbarbst : L^2_{0,q+1}(M)\to L^2_{0,q}(M)$ are continuous (with respect to
$\norm\cdot\norm_t$).
Additionally,\\ $\dbarbst G_{q,t}:H^s_{0,q}(M)\to H^s_{0,q-1}(M)$ and $G_{q,t}\dbarbst : H^s_{0,q+1}(M)\to H^s_{0,q}(M)$
are continuous if $t\geq T_s$.
\item The canonical solution operators for $\dbarbst$,
$\dbarb G_{q,t}:L^2_{0,q}(M)\to L^2_{0,q+1}(M)$
and $G_{q,t}\dbarb : L^2_{0,q-1}(M)\to L^2_{0,q}(M)$ are continuous (with respect to $\norm\cdot\norm_t$).
Additionally,\\ $\dbarb G_{q,t}:H^s_{0,q}(M)\to H^s_{0,q+1}(M)$
and $G_{q,t}\dbarb : H^s_{0,q-1}(M)\to H^s_{0,q}(M)$ are continuous if $t\geq T_s$.
\item The Szeg\"o projections $S_{q,t} = I - \dbarbst\dbarb G_{q,t}$ and $S_{q-1,t} = I - \dbarbst G_{q,t} \dbarb$ are continuous on
$L^2_{0,q}(M)$ and $L^2_{0,q-1}(M)$, respectively and with respect to $\norm\cdot\norm_t$. Additionally,
if $t\geq T_s$, then $S_{q,t}$ and $S_{q-1,t}$ are continuous on
$H^s_{0,q}$ and $H^s_{0,q-1}$, respectively.
\item  If $\tilde q = q$ or $q+1$ and
$\alpha\in H^s_{0,q}(M)$ so that  $\dbarb \alpha =0$ and $\alpha \perp \H^{\tilde q}_t$ (with respect to $\norm\cdot\norm_t$),
then there exists $u\in H^s_{0,\tilde q-1}(M)$ so that
\[
\dbarb u = \alpha;
\]
\item  If $\tilde q = q$ or $q+1$ and $\alpha\in C^\infty_{0,\tilde q}(M)$ satisfies $\dbarb\alpha=0$ and $\alpha \perp \H^{\tilde q}_t$
(with respect to $\la \cdot, \cdot \ra_t$),
then there exists $u\in C^\infty_{0,\tilde q-1}(M)$ so that
\[
\dbarb u = \alpha.
\]
\end{enumerate}
\end{thm}

\begin{rem} \label{rem:weighted vs. unweighted bound}
We will see below that the proof of Theorem \ref{thm:main theorem for unweighted} follows from Theorem
\ref{thm:main theorem for weighted spaces} and the fact that the weighted and unweighted norms are equivalent.
We will see in the proof of the main theorem that the constants improve as $t\to\infty$. In particular, we will show that
$\| \vp \|_t^2 \leq A_t \Qbt(\vp,\vp)$ where $A_t \to 0$ as $t\to\infty$.
A (weak) consequence is that
if the weight is strong enough, $\dbar$ and $\dbarbs$ have closed range in weighted $L^2$ with a constant that does not depend on the weight. In
the unweighted case, this means the constants may be quite large.
For a more quantitative discussion, see
Remark \ref{rem:weighted vs. unweighted bound 2} below.

Additionally, our results hold for any abstract CR-manifold for which a $q$-compatible function exists. $q$-compatible
functions are defined in Definition \ref{defn:compatible_functions}. They play the analogous role here of CR-plurisubharmonic functions
in \cite{Nic06,Rai10c}.
\end{rem}

In Section \ref{sec:def and not}, we introduce the notion of weak $Y(q)$ manifolds and $q$-compatible functions. In
Section \ref{sec:local coordinates}, we set up the microlocal analysis and build the weighted norm. Additionally, we
compute $\dbarb$ and $\dbarbs$ in local coordinates. In Section \ref{sec:basic estimate}, we adapt the microlocal
analysis in \cite{Nic06,Rai10c} and prove a basic estimate: Proposition \ref{prop:basic estimate}. In Section
\ref{sec:regularity theory}, we use the basic estimate to begin the study of the regularity theory for $\dbarb$, and we
prove Theorems \ref{thm:main theorem for weighted spaces} and \ref{thm:main theorem for unweighted} in Sections
\ref{sec:main theorem weighted} and \ref{sec:main theorem unweighted}, respectively.

%
%
\section{Definitions and Notation}\label{sec:def and not}

%
%
\subsection{CR manifolds and $\dbarb$}
\begin{defn}\label{defn:CR mfld}
Let $M\subset\C^N$ be a $C^\infty$ manifold of real dimension $2n-1$, $n\geq 2$. $M$ is called a \bfem{CR-manifold of hypersurface type} if
$M$ is equipped with a subbundle $T^{1,0}(M)$ of the complexified tangent bundle $\C TM = TM\otimes\C$ so that
\begin{enumerate}\renewcommand{\labelenumi}{(\roman{enumi})}
\item $\dim_{\C} \Toz(M)=n-1$;
\item $\Toz(M)\cap \Tzo(M) = \{0\}$ where $\Tzo(M) = \overline{\Toz}(M)$;
\item $\Toz(M)$ satisfies the following integrability condition: if $L_1,L_2$ are smooth sections of $\Toz(M)$, then so is the commutator $[L_1,L_2]$.
\end{enumerate}
\end{defn}

Since $M$ is a submanifold of $\C^N$, we can generate $\Toz_z(M)$  for $z\in M$ from  the
\bfem{induced CR-structure on $M$}  as
follows:  set
$\Toz_z(M) = \Toz_z(\C^N)\cap T_z(M)\otimes\C$ (under the natural inclusions). Since the complex dimension
of $\Toz_{z}(M)$ is $n-1$ for all $z\in M$, we can let $\Toz(M) =\bigcup_{z\in M} T_{z}^{1,0}(M)$. Observe that conditions
(ii) and (iii) are automatically satisfied in this case.

For the remainder of this article, $M^{2n-1}$ is a  smooth, orientable
CR-manifold of hypersurface type embedded in $\C^N$ for some $N\geq n$.
Let $\Lambda^{0,q}(M)$ be the bundle of $(0,q)$-forms on $M$, i.e., $\Lambda^{0,q}(M) = \bigwedge^q (\Tzo(M)^*)$. Denote the $C^\infty$ sections of
$\Lambda^{0,q}(M)$ by $C^\infty_{0,q}(M)$.


We construct $\dbarb$ using the fact that $M\subset\C^N$. There is a Hermitian inner product on $\Lambda^{0,q}(M)$ given by
\[
(\vp,\psi) = \int_M \langle \vp,\psi\rangle_x\, dV,
\]
where $dV$ is the volume element on $M$ and $\langle \vp,\psi \rangle_x$ is the induced inner product on $\Lambda^{0,q}(M)$.
This metric is compatible with the induced CR-structure, i.e., the vector spaces $\Toz_z(M)$ and $\Tzo_z(M)$
are orthogonal under the inner product.  The involution condition (iii) of Definition \ref{defn:CR mfld} means that $\bar\partial_b$ can be defined as the restriction of the de Rham exterior derivative $d$ to $\Lambda^{(0,q)}(M)$.  The inner product gives rise to an $L^2$-norm $\pnorm{\cdot}_0$, and we also denote the closure of $\dbarb$ in this norm by $\dbarb$ (by an abuse of notation).  In this way, $\dbarb : L^2_{0,q}(M)\to L^2_{0,q+1}(M)$ is a well-defined, closed, densely defined operator, and we define
$\dbarbs:L^2_{0,q+1}(M)\to L^2_{0,q}(M)$ to be the $L^2$-adjoint of $\dbarb$. The Kohn
Laplacian $\Boxb:L^2_{0,q}(M) \to L^2_{0,q}(M)$ is defined as
\[
\Boxb = \dbarbs\dbarb + \dbarb\dbarbs.
\]

%
%
\subsection{The Levi form and eigenvalue conditions}
The induced CR-structure has a local orthonormal basis $L_1,\dots, L_{n-1}$ for the $(1,0)$-vector fields in a neighborhood $U$ of each
point $x\in M$. Let $\omega_1,\dots,\omega_{n-1}$ be the dual basis of $(1,0)$-forms that satisfy $\langle \omega_j, L_k\rangle
=\delta_{jk}$. Then $\Lb_1,\dots, \Lb_{n-1}$ is a local orthonormal basis for the $(0,1)$-vector fields with dual basis $\ob_1,\dots,\ob_{n-1}$
in $U$. Also, $T(U)$ is spanned by $L_1,\dots,L_{n-1}$, $\Lb_1,\dots,\Lb_{n-1}$, and an additional vector field $T$ taken to be
purely imaginary (so $\bar T = -T$). Let $\gamma$ be the purely imaginary global $1$-form on $M$ that
annihilates $\Toz(M)\oplus \Tzo(M)$ and is normalized so that $\langle \gamma, T \rangle =-1$.

\begin{defn}\label{defn:Levi form}
The \bfem{Levi form at a point $x\in M$} is the Hermitian form given by $\langle d\gamma_x, L\wedge \bar L'\rangle$
where $L, L'\in \Toz_x(U)$, $U$ a neighborhood of $x\in M$.
\end{defn}

\begin{defn}\label{defn: pseudoconvex}
We call $M$ \bfem{weakly pseudoconvex} if there exists a
form $\gamma$ such that the Levi form is positive semi-definite at all $x\in M$ and \bfem{strictly pseudoconvex} if there
is a form $\gamma$ such that the Levi form is positive definite at all $x\in M$.
\end{defn}

The following two (standard) definitions are taken from  Chen and Shaw \cite{ChSh01}.
\begin{defn}\label{defn:Z(q) and Y(q)}Let $M$ be an oriented CR-manfiold of real dimension $2n-1$ with $n\geq 2$. $M$ is said to satisfy
\bfem{condition Z(q)}, $1\leq q \leq n-1$, if the Levi form associated with $M$ has at least $n-q$ positive eigenvalues
or at least $q+1$ negative eigenvalues at every boundary point. $M$ is said to satisfy
\bfem{condition Y(q)}, $1\leq q \leq n-1$
if the Levi form has at least either $\max\{n-q,q+1\}$ eigenvalues of the same sign of $\min\{n-q,q+1\}$ pairs of eigenvalues of
opposite signs at every point on $M$.
\end{defn}

Note that $Y(q)$ is equivalent to $Z(q)$ and $Z(n-1-q)$. The necessity of the symmetric requirements for $\dbarb$ at levels $q$ and $n-1-q$
stems from the  duality between $(0,q)$-forms and $(0,n-1-q)$-forms (see \cite{FoKo72} or \cite{RaSt08} for details).

$Z(q)$ and $Y(q)$ are classical conditions and natural extensions of strict pseudoconvexity. We wish, however, for an extension of weak pseudoconvexity.
Let $P\in M$ and $U$ be a special boundary neighborhood. Then there exists an orthonormal basis $L_1,\dots, L_{n-1}$ of $\Toz(U)$. By the Cartan formula (see
\cite{Bog91}, p.14),
\[
\la d\gamma, L_j\wedge \Lb_k \ra = - \la \gamma, [L_j,\Lb_k] \ra.
\]
If
\[
[L_j, \Lb_k] = c_{jk} T \mod \Toz(U)\oplus\Tzo(U),
\]
then $\la d\gamma, L_j\wedge \Lb_k \ra = c_{jk}$. For this reason,
the matrix $(c_{jk})_{1\leq j,k\leq n-1}$ is called the Levi form with respect to $L_1,\dots, L_{n-1}$.

By weakening the definition of $Z(q)$, we obtain:
\begin{defn}\label{defn:weak Z(q)} Let $M$ be a smooth, compact, oriented CR-manifold of hypersurface type of real dimension $2n-1$.
We say $M$ satisfies \bfem{Z(q) weakly at P} if there exists
\begin{enumerate}\renewcommand{\labelenumi}{(\roman{enumi})}
\item  a special boundary neighborhood $U\subset M$ containing $P$;
\item  an integer  $m=m(U)\neq q$;
\item an orthonormal basis $L_1,\dots,L_{n-1}$ of $\Toz(U)$ so that
 $\mu_1+\cdots+\mu_q-(c_{11}+\cdots+c_{mm})\geq 0$ on $U$, where $\mu_1, \dots, \mu_{n-1}$ are the eigenvalues of the Levi form in increasing order.
 \end{enumerate}
  We say that $M$ is \bfem{weakly Z(q)} if $M$
 is $Z(q)$ weakly at $P$ for all $P\in M$ and the condition $m>q$ or $m<q$ is independent of $U\subset M$.
 As above, $M$ satisfies \bfem{Y(q) weakly at P} if $M$ satisfies $Z(q)$ weakly at $P$ and $Z(n-1-q)$ weakly at $P$.
\end{defn}

To see that Definition \ref{defn:weak Z(q)}  generalizes condition $Z(q)$, choose coordinates diagonalizing $c_{jk}$ at $P$ so that $c_{jj}|_P=\mu_j$.  If the Levi-form has at least $n-q$ positive eigenvalues, then $\mu_q>0$, so we can let $m=q-1$ and obtain $\mu_1+\cdots+\mu_q-(c_{11}+\cdots+c_{mm})=\mu_q>0$ at $P$.  If the Levi-form has at least $q+1$ negative eigenvalues, then $\mu_{q+1}<0$, so we can let $m=q+1$ and obtain $\mu_1+\cdots+\mu_q-(c_{11}+\cdots+c_{mm})=-\mu_{q+1}>0$ at $P$.  In either case, the sum is strictly positive at $P$, so the estimate extends to a neighborhood $U$.

The preceding argument also shows that weak-$Z(q)$ is satisfied by domains where the Levi-form is locally diagonalizable and has at least $n-q$ non-negative eigenvalues or $q+1$ non-positive eigenvalues.  However, diagonalizability is not necessary.  Consider the hypersurface in $\mathbb{C}^5$ defined by
$\rho(z)=\Imm z_5+|z_3|^2+|z_4|^2+(\Rre  z_1)(|z_1|^2-2|z_2|^2)$.  Under the coordinates $L_j=\frac{\partial}{\partial z_j}-2i\frac{\partial\rho}{\partial z_j}\frac{\partial}{\partial z_5}$ and $T=2i\frac{\partial}{\partial z_5}+2i\frac{\partial}{\partial \z_5}$ the Levi-form looks like
\[
\begin{pmatrix}
  2\Rre z_1 & -z_2 & 0 & 0 \\
  -\z_2 & -2\Rre z_1 & 0 & 0 \\
  0 & 0 & 1 & 0 \\
  0 & 0 & 0 & 1 \\
\end{pmatrix}.
\]
We can compute the eigenvalues of this matrix in increasing order as
\[
\left\{-\sqrt{4(\Rre z_1)^2+|z_2|^2},\sqrt{4(\Rre z_1)^2+|z_2|^2},1,1\right\}.
\]
Since the corresponding eigenvectors are discontinuous at $P=0$, the Levi-form can not be diagonalized in a neighborhood of $P=0$.  In fact, we can not even continuously separate the positive and negative eigenspaces.  Let $q=2$ and $m=0$.  The sum of the two smallest eigenvalues is zero, so this domain satisfies weak $Z(2)$, which is equivalent to weak $Y(2)$ when $n=5$.

The signature of the Levi-form may also change locally.  If we let $\rho(z)=\Imm z_5+|z_2|^2+|z_3|^2+|z_4|^2+\Rre ((z_1)^2\z_1)$ with $L_j$ and $T$ as before, then we have a diagonal Levi-form with eigenvalues $\left\{2\Rre (z_1),1,1,1\right\}$.
When $\Rre (z_1)>0$, we have four positive eigenvalues.  When $\Rre (z_1)<0$, we have three positive and one negative eigenvalues.  Note that since we always have at least three positive eigenvalues, this satisfies the standard definition of $Y(2)$.  From the standpoint of weak $Z(2)$, we can take $m=0$ and obtain $\mu_1+\mu_2=2\Rre (z_1)+1>0$ near $P$, or we can take $m=1$ and obtain $\mu_1+\mu_2-c_{11}=(2\Rre (z_1)+1)-2\Rre (z_1)=1>0$, so either value of $m$ may work.  Hence, the appropriate value of $m$ need not be constant on $M$.  However, since we disallow $m=q$,
the condition $m<q$ or $m>q$ must be global.

If we can choose $m<q$ independent of the local neighborhood $U$,
then weak $Z(q)$ agrees with $(q-1)$-pseudoconvexity (see \cite{Zam08} for the definition on boundaries of domains and further references, or \cite{AhBaZa06} for generic CR submanifolds).  If $M$ satisfies weak $Z(1)$ for a choice of $m=0$, then
$M$ is simply a weakly pseudoconvex CR-manifold of hypersurface type.

\begin{rem}\label{rem:m for Z(q) and Z(n-1-q) may be different} For a CR-manifold $M$ that satisfies weak $Y(q)$, the $m$ that corresponds to
$Z(q)$ has no relation to the $m$ that corresponds to $Z(n-1-q)$.
To emphasize this,  we may use $m_q$ for the integer-valued function on $M$ that corresponds to weak $Z(q)$ and similarly for $m_{n-1-q}$ for
weak $Z(n-1-q)$.
\end{rem}

%
%
\subsection{$q$-compatible functions}
Let $\I_q = \{J= (j_1,\dots,j_q)\in\N^q : 1 \leq j_1 < \cdots < j_q \leq n-1\}$.

Let $\lam$ be a function defined near $M$ and define the 2-form
\begin{equation}\label{eqn:Theta^lam def}
\Theta^\lam = \frac 12\Big( \p_b\dbarb\lam - \dbarb\p_b\lam\Big) + \frac 12 \nu(\lambda)d\gamma.
\end{equation}
where $\nu$ is the real normal to $M$.
We will sometimes consider $\Theta^\lam$ to be the matrix $\Theta^\lam = (\Theta^\lam_{jk})$.

\begin{defn}
\label{defn:compatible_functions}
  Let $M$ be a smooth, compact, oriented CR-manifold of hypersurface type of real dimension $2n-1$ satisfying $Z(q)$ weakly at some point $P\in M$.  Let $\lambda$ be a smooth function near $M$.  We say $\lambda$ is
 \bfem{$q$-compatible}
 with $M$ at $P$ if there exists a special boundary neighborhood $U\subset M$ containing $P$, an integer $m_q=m_q(U)$ from weak $Z(q)$,
 an orthonormal basis $L_1,\dots,L_{n-1}$ of $\Toz(U)$, and a constant $B_\lambda>0$ satisfying
  \begin{enumerate}\renewcommand{\labelenumi}{(\roman{enumi})}
    \item $\mu_1+\cdots+\mu_q-(c_{11}+\cdots+c_{mm})\geq 0$ on $U$, where $\mu_1, \dots, \mu_{n-1}$ are the eigenvalues of the Levi form in increasing order.
    \item $b_1+\cdots+b_q-(\Theta_{11}+\cdots+\Theta_{mm})\geq B_{\lambda}$ on $U$ if $m<q$, where $b_1,\dots,b_{n-1}$ are the eigenvalues of $\Theta$ in increasing order.
    \item $b_{n-q}+\cdots+b_{n-1}-(\Theta_{11}+\cdots+\Theta_{mm})\leq-B_{\lambda}$ on $U$ if $m>q$.
  \end{enumerate}
\end{defn}

We call $B_\lam$ the positivity constant of $\lambda$.  Observe that if $M$ is pseudoconvex, $M$ satisfies Definition \ref{defn:weak Z(q)} for any $1\leq q\leq n-1$ and any orthonormal basis $L_1,\ldots,L_{n-1}$ by selecting $m=0$.  Hence, plurisubharmonic functions will be $q$-compatible with pseudoconvex domains for any $1\leq q\leq n-1$.

\begin{rem}
\label{rem:Z(q) positive explanation}
If $\lam = |z|^2$ then Proposition \ref{prop:dbar vs. dbar_b on M} below proves that
$\Theta = \p\dbar$ when tested against complex tangent vectors of $M$. Tested against such vectors, $\Theta^{|z|^2} = I$.
Since this is diagonal and all of the eigenvalue of $I$ are 1,
$b_1+\cdots+b_q-(\Theta_{11}+\cdots+\Theta_{mm})=q-m\geq 1$ if $q>m$ and $b_{n-q}+\cdots+b_{n-1}-(\Theta_{11}+\cdots+\Theta_{mm})=q-m\leq -1$
if $q<m$.  Hence, $\lambda = |z|^2$ is always a $q$-compatible function on $M$ with positivity constant $1$.
\end{rem}

\begin{rem}
Without the requirement that $\{L_1,\ldots,L_{n-1}\}$ are orthonormal, $\lam=|z|^2$ may not be a $q$-compatible function for all values of $m\neq q$.  For a given choice of non-orthonormal local coordinates, we can always define a local function which is $q$-compatible for all allowable $q$ and $m$, but there is no guarantee that such local functions could be made global.  Hence, if we remove the restriction that the local coordinates in Definition \ref{defn:compatible_functions} are orthonormal, we must also assume the existence of a global function which is $q$-compatible for all allowable choices of $q$ and $m$.
\end{rem}

\begin{rem}
  We note that if for every $B_\lambda>0$ there exists a $q$-compatible function $\lambda$ satisfying $0\leq\lambda\leq 1$ with positivity constant $B_\lambda$, then the methods of \cite{Rai10c} can be incorporated into our current paper to show that the complex Green operator is compact.  Such a condition is analogous to Catlin's Property $(P)$ \cite{Cat84}.
\end{rem}

In this article, constants with no subscripts may depend on $n$, $N$, $M$ but not any relevant $q$-compatible function.
Those constants will be denoted with an appropriate subscript. The constant $A$ will be reserved for
the constant in the construction of the pseudodifferential operator in Section \ref{sec:local coordinates}.

%
%
\section{Computations in Local Coordinates}\label{sec:local coordinates}

\subsection{Local coordinates and CR-plurisubharmonicity}

The following result is proved in \cite{Rai10c}.
\begin{prop}\label{prop:dbar vs. dbar_b on M}
Let $M^{2n-1}$ be a  smooth, orientable
CR-manifold of hypersurface type embedded in $\C^N$ for some $N\geq n$. If $\lambda$ is a smooth function near $M$,
$L\in \Toz(M)$, and $\nu$ is the real part of the complex normal to $M$,
then on $M$
\[
\left\la\frac12\Big(\p\dbar\lam-\dbar\p\lam\Big), L\wedge\Lb\right\ra
- \left\la\frac12\Big(\p_b\dbarb\lam-\dbarb\p_b\lam\Big), L\wedge\Lb\right\ra
= \frac 12\nu\{\lam\} \la d\gamma, L\wedge \bar L\ra
\]
\end{prop}

\subsection{Pseudodifferential Operators} We follow the setup for the microlocal analysis in \cite{Rai10c}.
Since $M$ is compact, there exists a finite cover $\{ U_\nu\}_{\nu}$ so each $U_{\nu}$ has a special boundary system and
can be parameterized by a hypersurface in $\C^n$ ($U_\nu$ may be shrunk as necessary). To set up the microlocal analysis,
we need to define the appropriate pseudodifferential operators on each $U_{\nu}$. Let
$\xi = (\xi_1,\dots,\xi_{2n-2},\xi_{2n-1}) = (\xi',\xi_{2n-1})$ be the coordinates in Fourier space so that
$\xi'$ is dual to the part of $T(M)$ in the maximal complex subspace (i.e., $\Toz(M)\oplus \Tzo(M)$) and
$\xi_{2n-1}$ is dual to the totally real part of $T(M)$, i.e.,the ``bad" direction $T$. Define
\begin{align*}
\Cp &= \{ \xi : \xi_{2n-1} \geq \frac 12 |\xi'| \text{ and } |\xi|\geq1\};\\
\Cm &= \{\xi : -\xi\in \Cp\};\\
\Co &= \{\xi : -\frac 34|\xi'| \leq \xi_{2n-1}\leq \frac 34 |\xi'|\} \cup \{\xi : |\xi|\leq 1\}.
\end{align*}
Note that $\Cp$ and $\Cm$ are disjoint, but both intersect $\Co$ nontrivially. Next, we define functions on
$\{|\xi| : |\xi|^2 =1\}$. Let
\begin{align*}
\psp(\xi) &= 1 \text{ when } \xi_{2n-1}\geq \frac 34|\xi'| \text{ and } \supp \psp \subset \{\xi : \xi_{2n-1}\geq \frac 12|\xi'|\}; \\
\psm(\xi) &= \psp(-\xi); \\
\pso(\xi) &\text{ satisfies } \pso(\xi)^2 = 1- \psp(\xi)^2 - \psm(\xi)^2.
\end{align*}
Extend $\psp$, $\psm$, and $\pso$ homogeneously outside of the unit ball, i.e., if $|\xi|\geq 1$, then
\[
\psp(\xi) = \psp(\xi/|\xi|),\ \psm(\xi) = \psm(\xi/|\xi|),\text{ and } \pso(\xi) = \pso(\xi/|\xi|).
\]
Also, extend $\psp$, $\psm$, and $\pso$ smoothly inside the unit ball so that $(\psp)^2+(\psm)^2 + (\pso)^2 =1$. Finally,
for a fixed constant $A>0$ to be chosen later, define for any $t>0$
\[
\pspl(\xi) = \psi(\xi/(t A)),\ \psml(\xi) = \psm(\xi/(t A)),\text{ and }\psol(\xi) = \pso(\xi/(t A)).
\]
Next, let $\Pspl$, $\Psml$, and $\Pso$ be the pseudodifferential operators of order zero with symbols
$\pspl$, $\psml$, and $\psol$, respectively. The equality $(\pspl)^2+(\psml)^2 + (\psol)^2 =1$ implies that
\[
\Pspla\Pspl + \Psola\Psol + \Psmla\Psml = Id.
\]
We will also have use for pseudodifferential operators  that ``dominate" a given pseudodifferential operator. Let
$\psi$ be cut-off function and $\tilde\psi$ be another cut-off function so that $\tilde\psi|_{\supp \psi} \equiv 1$. If $\Psi$ and $\tilde\Psi$ are pseudodifferential operators with symbols $\psi$ and $\tilde\psi$, respectively, then we say that
$\tilde\Psi$ dominates $\Psi$.

For each $U_\nu$,  we can define
$\Pspl$, $\Psml$, and $\Psol$ to act on functions or forms supported in $U_\nu$, so let $\Pspln$, $\Psmln$, and $\Psoln$
be the pseudodifferential operators of order zero defined on $U_\nu$, and let $\Cpn$, $\Cmn$, and $\Con$ be the regions of
$\xi$-space dual to $U_\nu$ on which the symbol of each of those pseudodifferential operators is supported. Then it follows that:
\[
\Psplan\Pspln + \Psolan\Psoln + \Psmlan\Psmln = Id.
\]
Let $\tPsplm$ and $\tPsmlm$ be pseudodifferential operators that dominate $\Psplm$ and $\Psmlm$, respectively
(where $\Psplm$ and $\Psmlm$ are defined on some $U_\mu$). If $\tCpm$ and $\tCmm$ are the supports of
$\tPsplm$ and $\tPsmlm$, respectively, then we can choose $\{U_\mu\}$, $\tpsplm$, and $\tpsmlm$ so that the following result holds.

\begin{lem}\label{lem: neighborhood intersection lemma}
Let $M$ be a compact, orientable, embedded CR-manifold. There is a finite open covering $\{U_\mu\}_\mu$ of $M$ so that
if $U_\mu, U_\nu\in \{U_\mu\}$ have nonempty intersection, then
there exists a diffeomorphism $\vartheta$ between $U_\nu$ and $U_\mu$ with Jacobian
$\Jt$ so that:
\begin{enumerate}\renewcommand{\labelenumi}{(\roman{enumi})}
\item $\tJt(\tCpm)\cap \Cmn = \emptyset$ and $\Cpn \cap \tJt(\tCmm) = \emptyset$ where $\tJt$ is the inverse of the transpose
of $\Jt$;

\item Let ${}^\vartheta\!\Psplm$, ${}^\vartheta\!\Psmlm$, and ${}^\vartheta\!\Psolm$ be the transfers of
$\Psplm$, $\Psmlm$, and $\Psolm$, respectively via $\vartheta$. Then on
$\{\xi : \xi_{2n-1} \geq \frac 45 |\xi'| \text{ and } |\xi|\geq (1+\ep)tA\}$, the principal symbol
of ${}^\vartheta\!\Psplm$ is identically 1, on
$\{\xi : \xi_{2n-1} \leq -\frac 45 |\xi'| \text{ and } |\xi|\geq (1+\ep)tA\}$, the principal symbol
of ${}^\vartheta\!\Psmlm$ is identically 1, and on
$\{\xi : -\frac 13 \xi_{2n-1} \geq \frac 13 |\xi'| \text{ and } |\xi|\geq (1+\ep)tA\}$, the principal symbol
of ${}^\vartheta\!\Psolm$ is identically 1, where $\ep>0$ can be very small;

\item Let ${}^\vartheta\!\tPsplm$, ${}^\vartheta\!\tPsmlm$ be the transfers via $\vartheta$ of
$\tPsplm$ and $\tPsmlm$, respectively. Then the principal symbol of ${}^\vartheta\!\tPsplm$ is identically 1 on
$\Cpn$ and the principal symbol of ${}^\vartheta\!\tPsmlm$ is identically 1 on $\Cmn$;

\item $\tCpm\cap\tCmm = \emptyset$.
\end{enumerate}
\end{lem}
We will suppress the left superscript $\vartheta$ as it should be clear from the context which pseudodifferential operator
must be transferred. The proof of this lemma is contained in Lemma 4.3 and its subsequent discussion in \cite{Nic06}.

If $P$ is any of the operators  $\Psplm$, $\Psmlm$, or $\Psolm$, then it is immediate that
\begin{equation}\label{eqn:inverse order zero t}
D^\alpha_\xi \sigma(P) = \frac{1}{|t|^\alpha}q_\alpha(x,\xi)
\end{equation}
for $|\alpha|\geq 0$, where $q(x,\xi)$ is bounded independently of $t$.

%
%
\subsection{Norms}
We have a volume form $dV$ on $M$, and we define the following inner products and norms on functions (with their natural
generalizations to forms). Let $\lam$ be a smooth function defined near $M$. We define
\[
(\phi,\vp)_\lam = \int_M \phi \bar\vp\, e^{-\lam}\, dV, \text{ and } \|\vp \|_\lam^2 = (\vp,\vp)_\lam
\]
In particular,
$(\phi,\vp)_0 = \int_M \phi \bar\vp\, dV$ and $\|\vp \|_0^2 = (\vp,\vp)_0$ are the standard (unweighted) $L^2$ inner product and norm.
If $\vp = \sum_{J\in\I_q} \vp_J \ob_J$, then we use the common shorthand $\|\vp \| = \sum_{J\in\I_q} \|\vp_J \|$ where $\| \cdot \|$ represents
any norm of $\vp$.

We also need a norm that is well-suited for the microlocal arguments. Let $\lp$ and $\lm$ be smooth functions defined near $M$.
Let $\{\zn\}$ be a partition of unity
subordinate to the covering $\{U_\nu\}$ satisfying $\sum_{\nu} \zn^2=1$. Also, for each $\nu$, let $\tzn$ be a cutoff function
that dominates $\zn$ so that $\supp\tzn \subset U_\nu$. Then we define the global inner product and norm as follows:
\begin{multline*}
\langle \phi,\vp \ralplm = \la \phi, \vp \rapm
= \sumn \Big[ ( \tzn\Pspln\zn \phi^\nu, \tzn\Pspln\zn \vp^\nu )_{\lp} \\
+ (\tzn \Psoln \zn \phi^\nu ,\tzn \Psoln \zn \vp^\nu )_0
+ (\tzn \Psmln \zn \phi^\nu, \tzn \Psmln \zn \vpn )_{\lm} \Big]
\end{multline*}
and
\[
\norm \vp \normlplm^2 = \norm \vp \normpm^2
= \sumn \Big[ \| \tzn\Pspln\zn \vp^\nu \|_{\lp}^2 + \|\tzn \Psoln \zn \vp^\nu \|_0^2
+ \|\tzn \Psmln \zn \vpn \|_{\lm}^2 \Big],
\]
where $\vpn$ is the form $\vp$ expressed in the local coordinates on $U_\nu$. The superscript $\nu$ will often be omitted.

For a form $\vp$ supported on $M$, the Sobolev norm of order $s$ is given by the following:
\[
\|\vp\|_s^2 = \sumn \|\tzn\Lambda^s\zn \vp^\nu \|_{0}^2
\]
where $\Lambda$ is defined to be the pseudodifferential operator with symbol $(1+|\xi|^2)^{1/2}$.

In \cite{Rai10c}, it is shown that  there exist constants $c_\pm$ and $C_\pm$ so that
\begin{equation}\label{eqn:norm equivalence}
c_\pm \| \vp\|_0^2 \leq \norm \vp \normlplm^2 \leq C_\pm \|\vp\|_0^2
\end{equation}
where $c_\pm$ and $C_\pm$ depend on $\max_{M}\{ |\lp| + |\lm|\}$ (assuming $tA\geq1$).
Additionally, there exists an invertible self-adjoint operator
$H_\pm$ so that $(\phi,\vp)_0 = \la \phi, H_\pm \vp \rapm$.

%
%
\subsection{$\dbarb$ and its adjoints}

If $f$ is a function on $M$, in local coordinates,
\[
\dbarb f = \sum_{j=1}^{n-1} \Lb_j f \ob_j,
\]
while if $\vp$ is a $(0,q)$-form, there exist functions $m_K^J$ so that
\[
\dbarb\vp = \sum_{\atopp{J\in\I_q}{K\in\I_{q+1}}} \sum_{j=1}^{n-1} \ep^{jJ}_K \Lb_j \vp_J\ob_K
+ \sum_{\atopp{J\in\I_q}{K\in\I_{q+1}}} \vp_J m^J_K \ob_K
\]
where $\ep^{jJ}_K$ is $0$ if $\{j\}\cup J \neq K$ as sets and is the sign of the permutation that reorders $jJ$ as $K$. We also define
\[
\vp_{jI} = \sum_{J\in\I_q} \ep^{jI}_J \vp_J
\]
(in this case, $|I|=q-1$ and $|J|=q$).
Let $\Lba_j$ be the adjoint of $\Lb_j$ in $ (\cdot\, , \cdot)_0$, $\Lbal_j$ be the adjoint of $\Lb_j$ in
$(\cdot\, ,\cdot)_\lam$. We define
$\dbarbs$ and $\dbarbsl$  in $L^2(M)$ and $L^2(M,e^{-\lam})$, respectively. In this paper, $\lambda$ stands for $\lp$ or $\lm$ and we will abbreviate
$\dbarb^{*,\lp}$ by $\dbarbsp$ and similarly for $\dbarbsm$, $\Lbap$, $\Lbam$, etc.

On a $(0,q)$-form $\vp$, we have (for some functions $f_j\in C^\infty(U)$)
\begin{align*}
\dbarbs \vp &= \sum_{I\in\I_{q-1}}\sum_{j=1}^{n-1}  \Lba_j \vp_{jI} \ob_I
+ \sum_{\atopp{I\in\I_{q-1}}{J\in\I_q}} \overline{m^I_J} \vp_J  \ob_I \\
&= -\sum_{I\in\I_{q-1}}\sum_{j=1}^{n-1}\big( L_j \vp_{jI} + f_j \vp_{jI}\big)\omb_I
+ \sum_{\atopp{I\in\I_{q-1}}{J\in\I_q}} \overline{m^I_J} \vp_J  \ob_I
\end{align*}
\begin{align} \label{eqn:dbarb adjoints}
\dbarbsl \vp &= \sum_{I\in\I_{q-1}}\sum_{j=1}^{n-1} \Lbal_j \vp_{jI} \ob_I
+ \sum_{I\in\I_{q-1}} \overline{m^I_J} \vp_J  \ob_I \\
&= -\sum_{I\in\I_{q-1}}\sum_{j=1}^{n-1}\big( L_j \vp_{jI} - L_j\lam \vp_{jI} + f_j \vp_{jI}\big)\omb_I
+ \sum_{\atopp{I\in\I_{q-1}}{J\in\I_q}} \overline{m^I_J} \vp_J  \ob_I \nn
\end{align}
Consequently, we see that
\[
\dbarbsl = \dbarbs - [\dbarbs,\lam],
\]
and both adjoints have the same domain. Finally, let $\dbarbspm$ be the adjoint of $\dbarb$ with respect to
$\la \cdot\, , \cdot \rapm$.

The computations proving Lemma 4.8 and Lemma 4.9 and equation (4.4) in \cite{Nic06}
can be applied here with only
a change of notation, so we have the following two
results, recorded here as Lemma \ref{lem: dbarbspm computation} and Lemma \ref{lem:energy form estimate for dbarbspm}.
The meaning of the results is that $\dbarbspm$ acts like $\dbarbsp$ for forms whose support is basically $\Cp$
and $\dbarbsm$ on forms whose support is basically $\Cm$.
\begin{lem}\label{lem: dbarbspm computation}
On smooth $(0,q)$-forms,
\begin{multline*}
\dbarbspm = \dbarbs - \summ \zm^2 \tPsplm[\dbarbs,\lp] + \summ\zm^2 \tPsmlm [\dbarbs,\lm] \\
+ \summ \Big( \tzm [\tzm\Psplm\zm,\dbarb]^*\tzm \Psplm\zm + \zm\Psplam\tzm[\dbarbsp,\tzm\Psplm\zm]\tzm \\
+ \tzm [\tzm\Psmlm\zm,\dbarb]^*\tzm \Psmlm\zm + \zm\Psplam\tzm[\dbarbsm,\tzm\Psmlm\zm]\tzm + E_A \Big),
\end{multline*}
where the error term $E_A$ is a sum of order zero terms and ``lower order" terms. Also, the symbol of $E_A$ is supported in
$\Com$ for each $\mu$.
\end{lem}

We are now ready to define the energy forms that we use. Let
\begin{align*}
\Qbpm(\phi,\vp) &= \la\dbarb\phi,\dbarb\vp\rapm + \la\dbarbspm\phi,\dbarbspm\vp\rapm \\
\Qbp(\phi,\vp)&= (\dbarb\phi,\dbarb\vp)_\lp + (\dbarbsp\phi,\dbarbsp\vp)_\lp \\
\Qbo(\phi,\vp) &= (\dbarb\phi,\dbarb\vp)_0 + (\dbarbs\phi,\dbarbs\vp)_0 \\
\Qbm(\phi,\vp) &= (\dbarb\phi,\dbarb\vp)_\lm + (\dbarbsm\phi,\dbarbsm\vp)_\lm.
\end{align*}

\begin{lem}\label{lem:energy form estimate for dbarbspm}
If $\vp$ is a smooth $(0,q)$-form on $M$, then there exist constants $K, K_\pm$ and  $K'$ with $K\geq 1$ so that
\begin{multline}\label{eqn:energy form -- dbarb commuted by psi-do}
K\Qbpm(\vp,\vp) + K_t \sumn \|\tzn\tPsoln\zn\vpn\|_0^2 +  K'\|\vp\|_0^2 + O_t (\|\vp\|_{-1}^2)
\geq \sumn \Big[ \Qbp(\tzn\Pspln\zn\vpn,\tzn\Pspln\zn\vpn) \\
+ \Qbo(\tzn\Psoln\zn\vpn,\tzn\Psoln\zn\vpn) + \Qbm(\tzn\Psmln\zn\vpn,\tzn\Psmln\zn\vpn) \Big]
\end{multline}
$K$ and $K'$ do not depend on $t, \lm$ or $\lp$.
\end{lem}

Also, since $\dbarbsl = \dbarbs +$ ``lower order" and $\Psi^\lam_{\mu,t}$ satisfies \eqref{eqn:inverse order zero t},
commuting $\dbarbsl$ by  $\Psi^\lam_{\mu,t}$ creates error
terms of order 0 that do not depend on $t$ or $\lam$, although lower order terms that may depend on $t$ and $\lam$.

%
%
\section{The Basic Estimate}\label{sec:basic estimate}
The goal of this section is to prove a basic estimate for smooth forms on $M$.
\begin{prop}\label{prop:basic estimate}
Let $M\subset \C^N$ be a compact, orientable CR-manifold of hypersurface type of dimension $2n-1$
and $1\leq q\leq n-2$.
Assume that $M$
admits  functions $\lam_1$ and $\lam_2$ where
$\lam_1$ is a $q$-compatible function and $\lam_2$ is an $(n-1-q)$-compatible function with positivity constants $B_{\lp}$ and $B_{\lm}$, respectively.
Let $\vp\in\Dom(\dbarb)\cap\Dom(\dbarbs)$.
Set
\[
\lp= \begin{cases} t\lam_1 & \textrm{if }m_q<q \\ -t\lam_1 & \textrm{if }m_q>q \end{cases}
\]
and
\[
\lm= \begin{cases} -t\lam_2 & \textrm{if }m_{n-1-q}<n-1-q \\ t\lam_2 & \textrm{if }m_{n-1-q}>n-1-q \end{cases}.
\]
There exist constants $K$,  $K_\pm$, and $K_\pm'$ where $K$  does
not depend on  $\lp$ and $\lm$ so that
\[
t B_\pm \norm \vp\normpm^2 \leq K \Qbpm(\vp,\vp) + K \norm \vp\normpm^2 +
K_\pm\sumn\sum_{J\in\I_q}\|\tzn\tPsoln\zn\vpn_J\|_0^2 + K_\pm' \|\vp\|_{-1}^2.
\]
The constant $B_\pm= \min\{ B_{\lp}, B_{\lm} \}$.
\end{prop}

 For Theorem \ref{thm:main theorem for unweighted}, we will use
$\lam_1=\lam_2 = |z|^2$.

\subsection{Local Estimates}

The crucial multilinear algebra that we need is contained in the following lemma from Straube \cite{Str09}:
\begin{lem}\label{lem:linear algebra}
Let $B = (b_{jk})_{1\leq j,k\leq n}$ be a Hermitian matrix and $1\leq q \leq n$. The following are equivalent:
\begin{enumerate}\renewcommand{\labelenumi}{(\roman{enumi})}
\item If $u\in \Lambda^{(0,q)}$, then $\displaystyle \sum_{K\in\I_{q-1}} \sum_{j,k=1}^n b_{jk} u_{jK} \overline{u_{kK}} \geq M |u|^2$.
\item The sum of any $q$ eigenvalues of $B$ is at least $M$.
\item $\displaystyle \sum_{s=1}^q \sum_{j,k=1}^n b_{jk} t^s_j \overline{t^s_k} \geq M$ whenever $t^1,\dots, t^q$ are orthonormal in $\C^n$.
\end{enumerate}
\end{lem}

We work on a fixed $U = U_\nu$. On this neighborhood, as above, there exists an orthonormal basis of vector fields
$L_1,\dots,L_{n}$, $\Lb_1,\dots,\Lb_n$ so that
\begin{equation}\label{eqn:L, Lb adjoint 1}
[L_j,\Lb_k] = c_{jk}T + \sum_{\ell=1}^{n-1} (d_{jk}^\ell L_\ell - \bar d_{kj}^\ell \Lb_\ell)
\end{equation}
if $1\leq j,k\leq n-1$, and $T = L_n-\Lb_n$.  Note that $c_{jk}$ are the coefficients of the Levi form.
Recall that $\Lbap$, $\Lba$, and $\Lbam$ are the adjoints of $\Lb$ in $(\cdot,\cdot)_\lp$, $(\cdot,\cdot)_0$,
and $(\cdot,\cdot)_\lm$, respectively. From \eqref{eqn:dbarb adjoints}, we see that
\[
\Lbal_j = -L_j+L_j\lam -f_j\,
\]
and plugging this into (\ref{eqn:L, Lb adjoint 1}),  we have
\begin{equation} \label{eqn:Lba, Lb commutator}
[\Lbal_j,\Lb_k] = -c_{jk}T + \sum_{\ell=1}^{n-1} \Big(d_{jk}^\ell (\Lbal_\ell-L_\ell\lam+f_\ell) + \bar d_{kj}^\ell \Lb_\ell\Big)
-\Lb_k L_j\lam +\Lb_kf_j.
\end{equation}

Because of Lemma \ref{lem:energy form estimate for dbarbspm}, we may turn our attention to the the quadratic
\[
\Qblp = (\dbarb\vp,\dbarb\vp)_\lam + (\dbarbsl\vp,\dbarbsl\vp)_\lam.
\]
We introduce the error term
\[
  E(\varphi)\leq C\left(\|\vp\|^2_\lam+\sum_{j=1}^{n-1}|(h\Lb_j\vp,\vp)_\lam|\right)=C\left(\|\vp\|^2_\lam+\sum_{j=1}^{n-1}|(\tilde h \Lbal_j\vp,\vp)_\lam|\right)
\]
where the operators $\Lb_j$ and $\Lbal_j$ act componentwise,  $C$ is a constant independent of $\vp$ and $\lam$, and
$h$ and $\tilde h$ are bounded functions that are independent of $t$, $A$, $\lp$, $\lm$, and the other quantities that are carefully minding.
Recall the definition that $\vp_{jK} = \sum_{J\in \I_q} \ep^{jK}_J \vp_J$.
As in the proof of Lemma 4.2 in \cite{Rai10c}, we compute that for smooth $\vp$ supported in
a special boundary neighborhood,
\begin{align}
&\Qblp = \sum_{J\in\I_q} \sum_{j=1}^{n-1} \|\Lb_j \vp_J \|_{\lam}^2+\sum_{I\in \I_{q-1}} \sum_{j,k=1}^{n-1}\Rre  \big(c_{jk}T\vp_{jI},\vp_{kI} \big)_\lam + E(\varphi) \nn \\
&+\sum_{I\in \I_{q-1}} \sum_{j,k=1}^{n-1}\Bigg\{\frac12 \big( (\Lb_j L_k \lam + L_j \Lb_k\lam)\vp_{jI}, \vp_{kI} \big)_\lam
+ \frac 12\sum_{\ell=1}^{n-1} \big( (d^{\ell}_{jk} L_\ell \lam + \overline{d^\ell_{jk}} \Lb_\ell \lam) \vp_{jI}, \vp_{kI} \big)_\lam \Bigg\} \label{eqn:Qblp basic}.
\end{align}

The weak $Z(q)$-hypothesis suggests that we ought to integrate by parts to take advantage of the positivity/negativity conditions.  By \eqref{eqn:Lba, Lb commutator} and integration by parts, we have
\begin{equation}
\label{eqn:int_by_parts}
  \|\Lb_j \vp_J\|^2_\lam-\|\Lbal_j\vp_J\|^2_\lam=-\Rre (c_{jj}T\vp_J,\vp_J)-\sum_{\ell=1}^{n-1}\Rre \left(d_{jj}^\ell(L_\ell\lam)\vp_J,\vp_J\right)-\Rre ((\Lb_j L_j\lam)\vp_J,\vp_J)+E(\vp).
\end{equation}

Consequently, we can use \eqref{eqn:Lba, Lb commutator}   and
\eqref{eqn:int_by_parts} to obtain
\begin{align}
&\Qblp = \sum_{J\in\I_q}\Big\{ \sum_{j=1}^{m} \|\Lbal_j \vp_J \|_{\lam}^2
+  \sum_{j=m+1}^{n-1} \|\Lb_j \vp_J \|_{\lam}^2 \Big\}+ E(\vp) \nn \\
&+\sum_{I\in \I_{q-1}} \sum_{j,k=1}^{n-1}\Rre  \big(c_{jk}T\vp_{jI},\vp_{kI} \big)_\lam-\sum_{J\in \I_q} \sum_{j=1}^m \Rre  \big(c_{jj}T\vp_J,\vp_J \big)_\lam  \nn \\
&+\sum_{I\in \I_{q-1}} \sum_{j,k=1}^{n-1}\Bigg\{\frac12 \big( (\Lb_j L_k \lam + L_j \Lb_k\lam)\vp_{jI}, \vp_{kI} \big)_\lam
+ \frac 12\sum_{\ell=1}^{n-1} \big( (d^{\ell}_{jk} L_\ell \lam + \overline{d^\ell_{jk}} \Lb_\ell \lam) \vp_{jI}, \vp_{kI} \big)_\lam \Bigg\}\label{eqn:Qblp after parts for Cp}\\
&-\sum_{J\in \I_q} \sum_{j=1}^m \Bigg\{ \frac12 \big( (\Lb_j L_j \lam + L_j \Lb_j\lam)\vp_J, \vp_J \big)_\lam
+ \frac 12\sum_{\ell=1}^{n-1} \big( (d^{\ell}_{jj} L_\ell \lam + \overline{d^\ell_{jj}} \Lb_\ell \lam) \vp_J, \vp_J \big)_\lam \Bigg\}.\nn
\end{align}

We are now in a position to control the ``bad" direction terms. Recall the following consequence of the sharp G{\aa}rding inequality from
\cite{Rai10c}.
\begin{prop}\label{prop:Garding 1}
Let $R$ be a first order pseudodifferential operator such that  $\sigma(R)\geq\kappa$ where $\kappa$ is some positive
constant and $(h_{jk})$ a hermitian matrix (that does not depend on $\xi$).
Then there exists a constant $C$ such that if the sum of any $q$ eigenvalues of $(h_{jk})$ is nonnegative, then
\[
\Rre \Big\{\sum_{I\in\I_{q-1}}\sum_{j,k=1}^{n-1}\big(h_{jk}R u_{jI}, u_{kI}\big)
 \Big \}
\geq \kappa  \Rre \sum_{I\in\I_{q-1}}\sum_{j,k=1}^{n-1} \big(h_{jk} u_{jI}, u_{kI}\big)
-C \|u\|^2,
\]
and if the
the sum of any collection of $(n-1-q)$ eigenvalues of $(h_{jk})$ is nonnegative, then
\begin{multline*}
\Rre \Big\{\sum_{J\in\I_q}\sum_{j=1}^{n-1} \big(h_{jj}R u_J, u_J\big)
-\sum_{I\in\I_{q-1}}\sum_{j,k=1}^{n-1}\big(h_{jk}R u_{jI}, u_{kI}\big)\Big \} \\
\geq \kappa  \Rre \Big\{\sum_{J\in\I_q}\sum_{j=1}^{n-1} \big(h_{jj} u_J, u_J\big)
-\sum_{I\in\I_{q-1}}\sum_{j,k=1}^{n-1} \big(h_{jk} u_{jI}, u_{kI}\big) \Big \} -C \|u\|^2 .
\end{multline*}
\end{prop}
Note that $(h_{jk})$ may be a matrix-valued function in $z$ but may not depend on $\xi$.

The following lemma is the analog of Lemma 4.6 in \cite{Rai10c}.
\begin{lem} \label{lem:T bound lp}
Let  $M$ be as in Theorem \ref{thm:main theorem for weighted spaces} and $\vp$ a $(0,q)$-form supported on $U$ so that
up to a smooth term $\hat\vp$ is supported in $\Cp$.
Let
\[
(h^+_{jk}) = (c_{jk}) - \delta_{jk}\frac 1q \sum_{\ell=1}^m c_{\ell\ell}.
\]
Then
\begin{multline*}
\Rre \Big\{\sum_{I\in\I_{q-1}}\sum_{j,k=1}^{n-1}  \big(h^+_{jk}T \vp_{jI}, \vp_{kI}\big)_\lam\Big\}\\
\geq t A  \Rre \Big\{\sum_{I\in\I_{q-1}}\sum_{j,k=1}^{n-1} \big(h^+_{jk} \vp_{jI}, \vp_{kI}\big)_\lam
\Big \}
- O(\|\vp\|_{\lam}^2) - O_{t} (\|\tzn\tPsol\vp\|_0^2).
\end{multline*}
where the constant in $ O(\|\vp\|_{\lam}^2)$ does not depend on $t$.
\end{lem}

\begin{proof}
Observe that the eigenvalues of $(h^+_{jk})$ are $\mu_j - \frac 1q \sum_{\ell=1}^m c_{\ell\ell}$, so the smallest possible sum of any $q$ eigenvalues
of $(h^+_{jk})$ is
\[
\mu_1 + \cdots +\mu_q -  \sum_{\ell=1}^m c_{\ell\ell} \geq 0.
\]
With this inequality in hand, we employ the argument of Proposition 4.6 from \cite{Rai10c} with the following changes. First, we replace
$c_{jk}$ with $h^+_{jk}$. Also, we replace the $A$ with $tA$ (for example, the sentence ``By construction, $\xi_{2n-1}\geq A$ in $\Cp$ \dots" gets replaced
by ``By construction, $\xi_{2n-1}\geq tA$ in $\Cp$ \dots").
\end{proof}
Observe that
\begin{multline}\label{eqn:hjk with T for Cp}
\sum_{I\in \I_{q-1}} \sum_{j,k=1}^{n-1} \Rre  \big(c_{jk}T\vp_{jI},\vp_{kI} \big)_\lam
-\sum_{J\in \I_q} \sum_{j=1}^m \Rre  \big(c_{jj}T\vp_J,\vp_J \big)_\lam=\\
\Rre \Big\{\sum_{I\in\I_{q-1}}\sum_{j,k=1}^{n-1}  \big(h^+_{jk}T \vp_{jI}, \vp_{kI}\big)_\lam
 \Big \}.
\end{multline}

Now that we can eliminate the $T$ terms, we turn to controlling the remaining terms.
%
%

\begin{prop} \label{prop:local results for Pspl}
Let $\vp\in\Dom(\dbarb)\cap\Dom(\dbarbs)$ be a $(0,q)$-form supported in $U$. Assume that
$\lam$ is a $q$-compatible function with positivity constant $B_{\lp}$. If $m<q$, choose $\lp = t\lam$ and if
$m>q$, choose $\lp = -t\lam$.
Then there exists a constant $C$ that is independent of $B_{\lp}$ so that
\[
\Qbp(\tz\Pspl\vp,\tz\Pspl\vp) + C\|\tz\Pspl\vp\|_\lp^2 + O_{t}(\|\tz\tPsol\vp\|_0^2)
\geq t B_{\lp} \|\tz\Pspl\vp\|_{\lp}^2.
\]
\end{prop}

\begin{proof}
Let
\[
s_{jk}^+ = \frac12 (\Lb_k L_j \lp + L_j \Lb_k\lp)+ \frac 12\sum_{\ell=1}^{n-1}(d^{\ell}_{jk} L_\ell \lp + \overline{d^\ell_{kj}} \Lb_\ell \lp)
\]
and
\[
r_{jk}^+ = s_{jk}^+
 - \frac 1q \delta_{jk}\sum_{\ell=1}^m s_{\ell \ell}
\]
In this case (\ref{eqn:Qblp after parts for Cp}) can be rewritten as
\begin{align*}
&\Qbp(\phi,\phi) = \sum_{J\in\I_q}\Big\{ \sum_{j=1}^{m} \|\Lbap_j \phi_J \|_{\lp}^2
+  \sum_{j=m+1}^{n-1} \|\Lb_j \phi_J \|_{\lp}^2 \Big\} + E(\varphi)\\
&+\sum_{I\in\I_{q-1}} \sum_{j,k=1}^{n-1} \Rre  \big( (r_{jk}^+ + h_{jk}^+ T)\phi_{jI},\phi_{kI}\big)_\lp.
\end{align*}
As noted in \cite{Nic06,Rai10c}, one can check that
if $L = \sum_{j=1}^{n-1} \xi_j L_j$ (where $\xi_j$ is constant), then
\[
\Big\la \frac12\big(\p_b\dbarb\lp - \dbarb\p_b\lp\big) , L\wedge\Lb\Big\ra
= \sum_{j,k=1}^{n-1} s_{jk}^+  \xi_j\bar\xi_k.
\]
This means that $s_{jk}^+  = \Theta_{jk}^+ -\frac 12\nu(\lp)c_{jk}$. Thus, if
\[
\Gamma^\lp_{jk} = \Theta^\lp_{jk}
- \frac 1q \delta_{jk}\sum_{\ell=1}^m \Theta^\lp_{\ell\ell}
\]
then
\begin{multline*}
\Qbp(\phi,\phi) = \sum_{J\in\I_q}\Big\{ \sum_{j=1}^{m} \|\Lbap_j \phi_J \|_{\lp}^2
+  \sum_{j=m+1}^{n-1} \|\Lb_j \phi_J \|_{\lp}^2 \Big\} + E(\varphi)\\
+\sum_{I\in\I_{q-1}} \sum_{j,k=1}^{n-1} \Rre  \Big(\big (\Gamma_{jk}^\lp + h_{jk}^+ (T-\frac 12\nu(\lp))\big )\phi_{jI},\phi_{kI}\Big)_\lp.
\end{multline*}

Next, we replace $\phi$ with $\tz\Pspl\vp$.
Since $\supp\tz\subset U'$, and
the Fourier transform of $\tz\Pspl\vp$ is supported in $\Cp$ up to a smooth term, we can use Lemma \ref{lem:T bound lp} to control the $T$ terms.
Therefore,
from (\ref{eqn:Qblp after parts for Cp}) and the form of $E(\varphi)$,   we have that
\begin{align*}
\Qbp(\tz\Pspl\vp,\tz\Pspl\vp)  &\geq (1-\ep) \sum_{J\in\I_q}\Big\{ \sum_{j=1}^{m} \|\Lbap_j \tz\Pspl\vp_J \|_{\lp}^2
+  \sum_{j=m+1}^{n-1} \|\Lb_j \tz\Pspl\vp_J \|_{\lp}^2 \Big\}  \\
&+\sum_{I\in \I_{q-1}} \sum_{j,k=1}^{n-1}  \Rre  \Big(\big (\Gamma_{jk}^\lp + h_{jk}^+ (t A-\frac 12\nu(\lp))\big)\tz\Pspl \vp_{jI},\tz\Pspl\vp_{kI} \Big)_\lp
 \\
 &- O(\|\tz \Pspl\vp\|_0^2)- O_{t} (\|\tzn\tPsol\vp\|_0^2) .
\end{align*}
If we choose $A\geq\frac{1}{2}|\nu(\lam)|$, then $t A -\frac 12 \nu(\lp) \geq 0$. Since the sum of any $q$ eigenvalues of
$(h_{jk}^+)$ is nonnegative, these terms are strictly positive.
If $m<q$, then
the sum of any $q$ eigenvalues of $\Gamma^\lp$ is the sum of $q$ eigenvalues  of $t\Theta^\lam$ minus the sum of the first $m$ diagonal terms of $t\Theta^\lam$.
If $m>q$, the sum of any $q$ eigenvalues of $\Gamma^\lp$ is  the sum of the first $m$ diagonal terms of $t\Theta^\lam $ minus
the sum of $q$ eigenvalues of of $t\Theta^\lam$. In either case,
by the $q$-compatibility of $\lam$, we know that this sum is at least $t B_{\lam^+}$ where
$B_{\lam^+}$ is the positivity constant of $\lam$.
By Lemma \ref{lem:linear algebra}, this means that
\[
\Qbp(\tz\Pspl\vp,\tz\Pspl\vp) + C \|\tz \Pspl\vp\|_0^2 + O_{t} (\|\tzn\tPsol\vp\|_0^2) \geq t B_{\lp} \|\tz\Pspl\vp\|_{\lp}^2.
\]
\end{proof}

%
%
Observe that the statement of Proposition \ref{prop:local results for Pspl} is independent of the choice of local coordinates $L_1,\ldots,L_{n-1}$ and $m\neq q$.  Hence, to handle the terms with support in $\Cm$, we may choose new local coordinates and a new value of $m$ so that Definitions \ref{defn:weak Z(q)} and \ref{defn:compatible_functions} hold with $(n-1-q)$ in place of $q$.  We again integrate (\ref{eqn:Qblp basic}) by parts and compute
\begin{align}
&\Qblp = \sum_{J\in\I_q}\Big\{ \sum_{j=1}^{m}  \|\Lb_j \vp_J \|_{\lam}^2
+  \sum_{j=m+1}^{n-1} \|\Lbal_j \vp_J \|_{\lam}^2 \Big\} + E(\vp) \nn \\
&+\sum_{I\in \I_{q-1}} \sum_{j,k=1}^{n-1}  \Rre  \big(c_{jk}T\vp_{jI},\vp_{kI} \big)_\lam-\sum_{J\in \I_q} \sum_{j=m+1}^{n-1} \Rre  \big(c_{jj}T\vp_J,\vp_J \big)_\lam \nn\\
&+\sum_{I\in \I_{q-1}} \sum_{j,k=1}^{n-1} \Bigg\{ \frac12 \big( (\Lb_j L_k \lam + L_j \Lb_k\lam)\vp_{jI}, \vp_{kI} \big)_\lam
+ \frac 12\sum_{\ell=1}^{n-1} \big( (d^{\ell}_{jk} L_\ell \lam + \overline{d^\ell_{jk}} \Lb_\ell \lam) \vp_{jI}, \vp_{kI} \big)_\lam \Bigg\}
\label{eqn:Qblp after parts for Cm}\\
&-\sum_{J\in \I_q} \sum_{j=m+1}^{n-1} \Bigg\{ \frac12 \big( (\Lb_j L_j \lam + L_j \Lb_j\lam)\vp_J, \vp_J \big)_\lam
+ \frac 12\sum_{\ell=1}^{n-1} \big( (d^{\ell}_{jj} L_\ell \lam + \overline{d^\ell_{jj}} \Lb_\ell \lam) \vp_J, \vp_J \big)_\lam \Bigg\}\nn
\end{align}

By the argument of Lemma \ref{lem:T bound lp}, we can also establish the following:
\begin{lem}\label{lem:T bound lm}
Let $M$ be as in Theorem \ref{thm:main theorem for weighted spaces} and
$\vp$ be a $(0,q)$-form supported on $U$ so that up to a smooth term, $\hat\vp$ is supported in $\Cm$.
Let
\[
(h^-_{jk}) = (c_{jk}) - \delta_{jk}\frac 1{n-1-q} \sum_{\ell=1}^m c_{\ell\ell}.
\]
Then
\begin{multline*}
\sum_{J\in\I_q} \sum_{j=1}^{n-1} \big(h^-_{jj}(-T) \vp_J, \vp_J\big)_\lam
- \sum_{I\in\I_{q-1}} \sum_{j,k=1}^{n-1}
 \big( h^-_{jk}(-T) \vp_{jI},\vp_{kI}\big)_\lam \\
\geq  t A\bigg(\sum_{J\in\I_q}\sum_{j=1}^{n-1} \big(h^-_{jj} \vp_J, \vp_J\big)_\lam
- \sum_{I\in\I_{q-1}} \sum_{j,k=1}^{n-1}
 \big( h^-_{jk} \vp_{jI},\vp_{kI}\big)_\lam \bigg)
+ O(\|\vp\|_{\lam}^2) + O_{t} (\|\tzn\tPsol\vp\|_0^2).
\end{multline*}
\end{lem}
In a similar fashion to \eqref{eqn:hjk with T for Cp}, we have the equality
\begin{multline}\label{eqn:hjk with T for Cm}
\sum_{J\in \I_q} \sum_{j=m+1}^{n-1} \Rre  \big(c_{jj}T\vp_J,\vp_J \big)_\lam
- \sum_{I\in\I_{q-1}} \sum_{j,k=1}^{n-1}
 \Rre  \big(c_{jk}T\vp_{jI},\vp_{kI} \big)_\lam
 \\
= \Rre \Big\{\sum_{J\in\I_q}\sum_{j=1}^{n-1}  \big(h^-_{jj}T \vp_J, \vp_J\big)_\lam
- \sum_{I\in\I_{q-1}} \sum_{j,k=1}^{n-1}
 \big( h^-_{jk}T \vp_{jI},\vp_{kI}\big)_\lam \Big \}.
\end{multline}

Applying these to the proof of Proposition \ref{prop:local results for Pspl}, we obtain
\begin{prop} \label{prop:local results for Psml}
Let $\vp\in\Dom(\dbarb)\cap\Dom(\dbarbs)$ be a $(0,q)$-form supported in $U$. Assume that
$\lam$ is an $(n-1-q)$-compatible function with positivity constant $B_{\lm}$. If $m>n-1-q$, choose $\lm =t\lam$ and if
$m<n-1-q$, choose $\lm =-t\lam$.
Then there exists a constant $C$ that is independent of $B_{\lm}$ so that
\[
\Qbm(\tz\Psml\vp,\tz \Psml\vp) + C\|\tz\Psml\vp\|_\lm^2 + O_ {t}(\|\tz\tPsol\vp\|_0^2)
\geq t B_{\lm} \|\tz\Psml\vp\|_{\lm}^2.
\]
\end{prop}

We are now ready to prove the basic estimate, Proposition \ref{prop:basic estimate}.
\begin{proof}[Proof  (Proposition \ref{prop:basic estimate})]
From \eqref{eqn:energy form -- dbarb commuted by psi-do}, there exist constants $K$, $K_\pm$ so that
\begin{multline*}
K\Qbpm(\vp,\vp) + K_\pm\sumn \|\tzn\tPsoln\zn\vpn\|_0^2 + K'\|\vp\|_0^2 + O_{\pm}(\|\vp\|_{-1}^2) \\
\geq \sumn \Big[ \Qbp(\tzn\Pspln\zn\vpn,\tzn\Pspln\zn\vpn)
 + \Qbm(\tzn\Psmln\zn\vpn,\tzn\Psmln\zn\vpn) \Big].
\end{multline*}
From Proposition \ref{prop:local results for Pspl} and Proposition \ref{prop:local results for Psml} it follows that
by increasing the size of $K$, $K_\pm$, and $K'$
\[
K\Qbpm(\vp,\vp) + K_\pm\sumn \|\tzn\tPsoln\zn\vpn\|_0^2 +  K'\|\vp\|_0^2 + O_{\pm}(\|\vp\|_{-1}^2)
\geq t B_\pm \|\vp\|_0^2
\]
where $B_\pm = \min\{B_\lm,B_\lp\}$.
\end{proof}

%
\subsection{A Sobolev estimate in the ``elliptic directions"}
For forms whose Fourier transforms are supported up to a smooth term in $\Co$, we have better estimates.  The following
results are in \cite{Nic06,Rai10c}.

\begin{lem}\label{lem: Co supported terms are benign}
Let $\vp$ be a $(0,1)$-form supported in $U_\nu$ for some $\nu$ such that up to a smooth term,
$\hat\vp$ is supported in $\tCon$. There exist positive constants $C>1$ and $C_1>0$ so that
\[
C \Qbpm(\vp,H_{\pm} \vp) + C_1\|\vp\|_0^2 \geq \|\vp\|_1^2.
\]
\end{lem}
The proof in \cite{Nic06} also holds  at level $(0,q)$.

We can use 
Lemma \ref{lem: Co supported terms are benign}
to control terms of the form $\|\tzn \Psoln \zn\vpn\|_0^2$.
\begin{prop}\label{prop:controlling the elliptic terms}
For any $\ep>0$, there exists $C_{\ep,\pm}>0$ so that
\[
\|\tzn \Psoln \zn\vpn\|_0^2 \leq \ep \Qbpm(\vpn,\vpn) + C_{\ep,\pm} \|\vpn\|_{-1}^2.
\]
\end{prop}
See \cite{Rai10c} for a proof of this proposition.

%
%
\section{Regularity Theory for $\dbarb$}
\label{sec:regularity theory}

%
%
\subsection{Closed range for $\Boxbpm$.}

For $1\leq q\leq n-2$, let
\begin{align*}
\Hpm^q &= \{\vp\in\Dom(\dbarb)\cap\Dom(\dbarbs) : \dbarb\vp=0, \dbarbspm\vp=0\} \\
& = \{\vp\in\Dom(\dbarb)\cap\Dom(\dbarbs):\Qbpmp=0\}
\end{align*}
be the space of $\pm$-harmonic $(0,q)$-forms.

\begin{lem}\label{lem:Qbpmp controls norm vp}
Let $M^{2n-1}$ be a smooth, embedded CR-manifold of hypersurface type that admits a $q$-compatible function $\lp$ and an $(n-1-q)$-compatible function
$\lm$. If
$t>0$ is suitably large and $1\leq q \leq n-2$, then
\begin{enumerate}\renewcommand{\labelenumi}{(\roman{enumi})}
\item $\Hpm^q$ is finite dimensional;
\item There exists $C$ that does not depend on $\lp$ and $\lm$ so that
for all $(0,q)$-forms $\vp\in\Dom(\dbarb)\cap\Dom(\dbarbs)$ satisfying $\vp\perp\Hpm^q$ (with respect to  $\la\cdot,\cdot\rapm$) we have
\begin{equation}\label{eqn:Qbpmp controls norm}
\norm\vp\normpm^2 \leq C\Qbpmp.
\end{equation}
\end{enumerate}
\end{lem}

\begin{proof}
For $\vp\in\Hpm$, we can use Proposition \ref{prop:basic estimate} with $t$ suitably large (to absorb terms) so that
\[
tB_\pm\norm\vp\normpm^2 \leq C_{\pm}\big( \sumn \|\tzn\Psoln\zm\vpn\|_0^2 + \|\vp\|_{-1}^2\big).
\]
Also, by Proposition \ref{prop:controlling the elliptic terms},
\[
\sumn \|\tzn\Psoln\zm\vpn\|_0^2 \leq C_{\pm} \|\vp\|_{-1}^2.
\]
since $\Qbpm(\vp,\vp)=0$.
The unit ball in $\Hpm\cap L^2(M)$ is compact, and hence finite dimensional.

Assume that \eqref{eqn:Qbpmp controls norm} fails. Then there exists $\vp_k \perp\Hpm$
with $\norm\vp_k\normpm =1$ so that
\begin{equation}\label{eqn:k Qbpm inequality}
\norm\vp_k\normpm^2 \geq k\Qbpm(\vp_k,\vp_k).
\end{equation}
For $k$ suitably large, we can use Proposition \ref{prop:basic estimate} and the above argument to absorb $\Qbpm(\vp_k,\vp_k)$ by
$B_\pm \norm \vp_k\normpm$ to get:
\begin{equation}\label{eqn:norm controlled by -1 norm}
\norm\vp_k\normpm^2 \leq C_{\pm}\|\vp_k\|_{-1}^2.
\end{equation}
Since $L^2(M)$ is compact in $H^{-1}(M)$, there exists a subsequence $\vp_{k_j}$ that converges in $H^{-1}(M)$. However,
\eqref{eqn:norm controlled by -1 norm} forces $\vp_{k_j}$ to converge in $L^2(M)$ as well. Although the norm
$(\Qbpm(\cdot,\cdot) + \norm\cdot\normpm^2)^{1/2}$ dominates the $L^2(M)$-norm, \eqref{eqn:k Qbpm inequality} applied to
$\vp_{j_k}$ shows that  $\vp_{j_k}$
converges in the $(\Qbpm(\cdot,\cdot) + \norm\cdot\normpm^2)^{1/2}$ norm as well. The limit $\vp$
satisfies $\norm\vp\normpm=1$ and $\vp\perp\Hpm$. However, a consequence of \eqref{eqn:k Qbpm inequality} is that $\vp\in\Hpm$. This is a contradiction
and \eqref{eqn:Qbpmp controls norm} holds.
\end{proof}

Let
\[
\Hpmp^q = \{\vp\in L^2_{0,q}(M) : \la \vp,\phi\rapm =0,\text{ for all }\phi\in\Hpm^q\}.
\]
On $\Hpmp^q$, define
\[
\Boxbpm = \dbarb\dbarbspm + \dbarbspm\dbarb.
\]
Since $\dbarbspm = H_{\pm}\dbarbs + [\dbarbs,H_{\pm}]$, $\Dom(\dbarbspm) = \Dom(\dbarbs)$. This causes
\[
\Dom(\Boxbpm) = \{\vp\in L^2_{0,q}(M) : \vp\in\Dom(\dbarb)\cap\Dom(\dbarbs),\ \dbarb\vp\in\Dom(\dbarbs),\text{ and }
\dbarbs\vp\in\Dom(\dbarb)\}.
\]

%
%
\section{Proof of Theorem \ref{thm:main theorem for weighted spaces}.}\label{sec:main theorem weighted}
\subsection{Closed range in $L^2$.}
From Remark \ref{rem:Z(q) positive explanation}, we know that $|z|^2$ is a $q$-compatible functions with a positivity constant
of $1$. Thus, for suitably large $t$, the space of harmonic $(0,q)$-forms $\H_t^q := \Hpm^q$ is finite dimensional. Moreover,
if we use $\la \cdot, \cdot \ra_t$ for $\la \cdot, \cdot \rapm$ and $Q_{b,t}$ for $\Qbpm$,  then for
$\vp\perp\H_t^q$ (with respect to  $\la\cdot,\cdot\ra_t$)
\begin{equation}\label{eqn:Qbt controls norm}
\norm\vp\norm_t^2 \leq C Q_{b,t}(\vp,\vp).
\end{equation}
From H\"ormander \cite{Hor65}, Theorem 1.1.2, \eqref{eqn:Qbt controls norm} is equivalent to the closed range of
$\dbarb:L^2_{0,q}(M)\to L^2_{0,q+1}(M)$ and $\dbarbst : L^2_{0,q}(M)\to L^2_{0,q-1}(M)$ where both operators are defined with respect to
$\la\cdot,\cdot\ra_t$. By H\"ormander \cite{Hor65}, Theorem 1.1.1, this means that $\dbarbst : L^2_{0,q+1}(M)\to L^2_{0,q}(M)$ and
$\dbarb:L^2_{0,q-1}(M)\to L^2_{0,q}(M)$ also have closed range. Thus, the Kohn Laplacian $\Box_{b,t}$ on $(0,q)$-forms
also has closed range and $G_{q,t}$ exists and is a continuous operator on $L^2_{0,q}(M)$.

%
%
\subsection{Hodge theory and the canonical solutions operators.}
We now prove the existence of a  Hodge decomposition and the existence of the canonical solution operators. Unlike the standard computations for the
$\dbar$-Neumann operators and complex Green operators in the pseudoconvex case, we only have the existence of
the complex Green operator $G_{q,t}$ at a fixed level $q$ and not for all $1\leq q \leq n-1$. (hence, we cannot
commute $G_{q,t}$ with either $\dbarb$ or $\dbarbst$). If $H^q_t$ is the projection of $L^2_{0,q}(M)$ onto  $\H^q_t = \Null(\dbarb)\cap\Null(\dbarbst) =
\{ \vp\in L^2_{0,q}(M)\cap \Dom(\dbarb)\cap\Dom(\dbarbst) : \Qbt(\vp,\vp)=0\}$, then we know
\[
\vp = \dbarb\dbarbst G_{q,t}\vp + \dbarbst\dbarb G_{q,t}\vp + H^q_t \vp.
\]
We now find the canonical solution operators. Let $\vp$ be a $\dbarb$-closed $(0,q)$-form that  is orthogonal to $\H^q_t$. Then $H^q_t \vp=0$, so
\[
\vp = \dbarb\dbarbst G_{q,t}\vp + \dbarbst\dbarb G_{q,t}\vp.
\]
We claim that $\dbarbst\dbarb G_{q,t}\vp=0$. Following \cite{Nic06}, we note that
\[
0 = \dbarb\vp = \dbarb \dbarbst\dbarb G_{q,t}\vp,
\]
so
\[
0 = \la \dbarb \dbarbst\dbarb G_{q,t}\vp, \dbarb G_{q,t}\vp \ra_t = \norm \dbarbst\dbarb G_{q,t}\vp \norm_t^2.
\]
Thus, $\dbarbst\dbarb G_{q,t}\vp=0$ and the canonical solution operator to $\dbarb$ is given by $\dbarbst G_{q,t}$. A similar argument shows that
the canonical solution operator for $\dbarbst$ is given by $\dbarb G_{q,t}$.

In this paragraph, we will assume that all forms are perpendicular to $\H^q_t$.
For $\vp\in\Dom(\Box_{b,t})$, it follows that
\[
\vp = G_{q,t}\Box_{b,t} \vp = \Box_{b,t} G_{q,t} \vp.
\]
We will show that
\begin{equation}\label{eqn:commuting dbarb dbarbst by G_q}
\dbarb\dbarbst G_{q,t} = G_{q,t}\dbarb\dbarbst\qquad \text{and} \qquad\dbarbst\dbarb G_{q,t} = G_{q,t}\dbarbst\dbarb.
\end{equation}
Observe that
\begin{align}\label{eqn:dbarb a=0 implies dbarb dbarst G a = G dbarb dbarst a}
\dbarb\alpha &= 0 \Longrightarrow \alpha = \dbarb\dbarbst G_{q,t} \alpha = G_{q,t}\dbarb\dbarbst \alpha \\
\intertext{and}
\dbarbst\beta &= 0 \Longrightarrow \beta =\dbarbst \dbarb G_{q,t} \beta = G_{q,t}\dbarbst \dbarb \beta.
\label{eqn:dbarbst a=0 implies dbarst dbarb G a = G dbarst dbarb a}
\end{align}
Next, we claim that
\begin{equation}\label{eqn:dbarb f=0 implied dbarb G_q f=0}
\dbarb\vp =0 \Longrightarrow \dbarb G_q \vp =0
\end{equation}
and
\begin{equation}\label{eqn:dbarstb f=0 implied dbarbst G_q f=0}
\dbarbst\vp =0 \Longrightarrow \dbarbst G_q \vp =0.
\end{equation}
Indeed, we have that $\vp\perp \H^q_t$, so $\vp = \dbarb\dbarbst G_{q,t}\vp + \dbarbst\dbarb G_{q,t} \vp$. Since
$\Ran \dbarbst \perp \Null\dbarb$, $\dbarb\vp=0$ implies that $\dbarbst\dbarb G_{q,t} \vp=0$.
Since $\Ran(\dbarb)\perp \Null(\dbarbst)$, $\dbarbst\dbarb G_{q,t}\vp=0$ implies $\dbarb G_{q,t}\vp=0$, as desired. A similar argument
shows (\ref{eqn:dbarstb f=0 implied dbarbst G_q f=0}). To show \eqref{eqn:commuting dbarb dbarbst by G_q}, observe that
we can write $\vp = \alpha + \beta$ where $\dbarb\alpha=0$ and $\dbarbst\beta=0$. Thus, by
(\ref{eqn:dbarb a=0 implies dbarb dbarst G a = G dbarb dbarst a}) and (\ref{eqn:dbarstb f=0 implied dbarbst G_q f=0}),
\[
\dbarb\dbarbst G_{q,t} \vp = \dbarb\dbarbst G_{q,t} (\alpha+\beta) = \dbarb\dbarbst G_{q,t} \alpha
= G_{q,t} \dbarb\dbarbst \alpha= G_{q,t} \dbarb\dbarbst \vp.
\]
A similar argument with \eqref{eqn:dbarbst a=0 implies dbarst dbarb G a = G dbarst dbarb a}
and \eqref{eqn:dbarb f=0 implied dbarb G_q f=0} proves that
$\dbarbst\dbarb G_{q,t} \vp  = G_{q,t} \dbarbst\dbarb  \vp$, finishing the proof of (\ref {eqn:commuting dbarb dbarbst by G_q}).

%
%
\subsection{Closed range of $\dbarb:H^s_{0,q}(M)\to H^s_{0,q+1}(M)$ and $\dbarbst:H^s_{0,q}(M)\to H^s_{0,q-1}(M)$.}
We start with an argument to show closed range of $\dbarb:H^s_{0,q}(M)\to H^s_{0,q+1}(M)$ and $\dbarbst:H^s_{0,q}(M)\to H^s_{0,q-1}(M)$. Combining
Proposition \ref{prop:basic estimate} and Lemma \ref{lem: Co supported terms are benign}, if $t$ is sufficiently large, then
\begin{align*}
&\norm \Lambda^s \vp \norm_t^2 \leq \frac Ct \big( \norm \dbarb\Lambda^s \vp \norm_t^2 + \norm \dbarbst \Lambda^s \vp \norm_t^2 \big)
+ C_t \| u\|_{s-1}^2 \\
&\leq \frac Ct \big( \norm \Lambda^s\dbarb \vp \norm_t^2 + \norm \Lambda^s\dbarbst \vp \norm_t^2
+ \norm [\dbarb,\Lambda^s] \vp \norm_t^2 + \norm [\dbarbst, \Lambda^s] \vp \norm_t^2 \big)
+ C_t \| \vp\|_{s-1}^2 .
\end{align*}
As a consequence of  Lemma \ref{lem: dbarbspm computation}, $[\dbarbst, \Lambda^s] = P_{s} + t P_{s-1}$ where $P_s$ and $P_{s-1}$ are
pseudodifferential operators of order $s$ and $s-1$, respectively. Additionally, $[\dbarb,\Lambda^s]$ is a pseudodifferential operator of order $s$.
Consequently,
\[
\norm \Lambda^s \vp \norm_t^2 \leq \frac Ct \big( \norm \Lambda^s\dbarb \vp \norm_t^2 + \norm \Lambda^s\dbarbst \vp \norm_t^2
+ \norm \Lambda^s \vp \norm_t^2 \big) + C_t \| \vp\|_{s-1}^2.
\]
Choosing $t$ large enough and $\vp\in H^s_{0,q}(M)$ allows us to absorb terms to prove
\begin{multline*}
\|\vp\|_s^2  = \|\Lambda^s \vp\|_{0}^2 \leq C_t \norm \Lambda^s \vp \norm_t^2
\leq C_t \big( \norm \Lambda^s\dbarb \vp \norm_t^2 + \norm \Lambda^s\dbarbst \vp \norm_t^2  +  \| \vp\|_{s-1}^2  \big)  \\
\leq C_t \big( \|\dbarb \vp\|_s^2 + \|\dbarbst \vp\|_s^2  + \| \vp\|_{s-1}^2  \big).
\end{multline*}
Thus,  $\dbarb:H^s_{0,q}(M)\to H^s_{0,q+1}(M)$ and $\dbarbst:H^s_{0,q}(M)\to H^s_{0,q-1}(M)$ have closed range.

%
%
\subsection{Continuity of the complex Green's operator in $H^s_{0,q}(M)$.}
We now turn to the harder problem of showing continuity of the complex Green operator $\Gqdt$ in $H^s_{0,q}(M)$, $s>0$.
We use an elliptic regularization argument.
Let $\Qbt^\delta(\cdot,\cdot)$
be the quadratic form on $H^1_{0,q}(M)$ defined by
\[
\Qbt^\delta(u,v) = \Qbt(u,v) + \delta Q_{d_b}(u,v)
\]
where $Q_{d_b}$ is the hermitian inner product associated to the de Rham exterior derivative $d_b$, i.e.,
$Q_{d_b}(u,v) = \la d_bu,d_bv \ra_t + \la d_b^*u,d_b^*v\ra_t$. The inner product $Q_{d_b}$ has form domain
$H^1_{0,q}(M)$. Consequently, $\Qbt^\delta$ gives rise to a unique, self-adjoint, elliptic operator $\Box_{b,t}^\delta$ with
inverse $\Gqdt$.

From  Proposition \ref{prop:basic estimate} and Lemma \ref{lem: Co supported terms are benign}, if $t$ is large enough, then for
$\vp\in\Dom(\dbarb)\cap\Dom(\dbarbst)$, we have the estimate
\begin{equation}\label{eqn:good t basic estimate}
\norm \vp \norm_t^2 \leq \frac{K}{t} \Qbt(\vp,\vp) + C_t \|\vp\|_{-1}^2.
\end{equation}
Now let $\vp\in H^s_{0,q}(\Omega)$. Since $\Box_{b,t}^\delta$ is elliptic, $\Gqdt\vp\in H^{s+2}_{0,q}(M)$. Then
\begin{equation}\label{eqn:Gqdelta s to L^2}
\| \Gqdt\vp\|_s^2 = \| \Lambda^s \Gqdt \vp \|_0^2 \leq C_t \norm \Lambda^s \Gqdt \vp \norm_t^2.
\end{equation}
We now concentrate on finding a bound  for $\norm \Lambda^s \Gqdt \vp \norm_t^2$ that is independent of $\delta$. By (\ref{eqn:good t basic estimate}),
\begin{equation}\label{eqn:Lambda^s G_q bound with Q}
\norm \Lambda^s \Gqdt \vp \norm_t^2 \leq \frac Kt \Qbt(\Lambda^s \Gqdt \vp,\Lambda^s \Gqdt \vp) + C_{t,s} \|\Gqdt\vp\|_{s-1}^2.
\end{equation}
Observe that if $(\Lambda^s)^{*,t}$ is the adjoint of $\Lambda^s$ under the inner product $\la\cdot,\cdot\ra_t$, then
\[
\la\Lambda^s u, v\ra_t = (u,\Lambda^s H_t^{-1} v)_0 = \la u,H_t \Lambda^s H_t^{-1} v\ra_t
= \la u, (\Lambda^s + [H_t,\Lambda^s]H_t^{-1})v \ra_t
\]
implies that $(\Lambda^s)^{*,t} = \Lambda^s + [H_t,\Lambda^s]H_t^{-1}$. Therefore,
it is a standard consequence of Lemma 3.1 in \cite{KoNi65}
(or Lemma 2.4.2 in \cite{FoKo72}) that
\begin{align}
\Qbt(\Lambda^s \Gqdt \vp,\Lambda^s \Gqdt \vp) \leq \Qbt^\delta (\Lambda^s \Gqdt \vp,\Lambda^s \Gqdt \vp)
&\leq |\la \Lambda^s\vp, \Lambda^s \Gqdt \vp \ra_t| + C \| \Gqdt\vp \|_s^2 + C_{t,s} \| \Gqdt \vp\|_{s-1}^2\nn\\
&\leq  \norm \Lambda^s \vp \norm_t \norm \Lambda^s \Gqdt \vp \norm_t + \| \Gqdt \vp \|_{s-1}^2 \nn \\
&\leq K_t  \|\vp\|_s^2 + C \norm \Lambda^s \Gqdt \vp \norm_t^2 + C_{t,s} \| \Gqdt \vp\|_{s-1}^2
\label{eqn:good bound for Qbt(lam^s G vp)}
\end{align}
where $C>0$ does not depend on $\delta$ or $t$.

Plugging (\ref{eqn:good bound for Qbt(lam^s G vp)}) into
(\ref{eqn:Lambda^s G_q bound with Q}), we see that
\[
\norm \Lambda^s \Gqdt \vp \norm_t^2 \leq \frac Kt \Big(  K_t  \|\vp\|_s^2 + C \norm \Lambda^s \Gqdt \vp \norm_t^2  \Big)
+  C_{t,s} \|\Gqdt\vp\|_{s-1}^2.
\]
If $t$ is sufficiently large, then it follows that
\begin{equation}\label{eqn:good bound for lambda^s Gqdt vp}
\norm \Lambda^s \Gqdt \vp \norm_t^2 \leq  K_t  \|\vp\|_s^2 +  C_{t,s} \|\Gqdt\vp\|_{s-1}^2
\end{equation}
since $\norm \Lambda^s \Gqdt \vp \norm_t^2<\infty$   (recall that $\Gqdt \vp \in H^{s+2}_{0,q}(M)$).
Plugging (\ref{eqn:good bound for lambda^s Gqdt vp}) into (\ref{eqn:Gqdelta s to L^2}), we have the bound
\begin{equation}\label{eqn:Gqdt bound, delta gone}
\| \Gqdt\vp\|_s^2 \leq K_t  \|\vp\|_s^2 +  C_{t,s} \|\Gqdt\vp\|_{s-1}^2.
\end{equation}
We now turn to letting $\delta\to 0$. Observe that $K_t$ and $C_{t.s}$ are independent of $\delta$. We have shown that if
$\vp\in H^s_{0,q}(M)$, then $\{ \Gqdt \vp: 0 <\delta <1\}$ is bounded in $H^s_{0,q}(M)$. Thus, there exists a sequence $\delta_k\to 0$
and $\tilde u \in H^s_{0,q}(M)$ so that $G_{q,t}^{\delta_k}u \to \tilde u$ weakly in $H^s_{0,q}(M)$. Consequently, if $v\in H^{s+2}_{0,q}(M)$,
then
\[
\lim_{k\to\infty} \Qbt^{\delta_k} (G_{q,t}^{\delta_k} u, v) = \Qbt(\tilde u,v).
\]
However,
\[
\Qbt^{\delta_k}(G_{q,t}^{\delta_k} u,v) = (u,v)= \Qbt(G_{q,t} u,v),
\]
so $G_{q,t}u = \tilde u$ and (\ref{eqn:Gqdt bound, delta gone}) is satisfied with $\delta=0$. Thus, $G_{q,t}$ is a continuous operator on
$H^s_{0,q}(M)$.

%
%
\subsection{Continuity of the canonical solution operators in $H^s_{0,q}(M)$.}
Continuity of $\dbarb G_{q,t}$ and $\dbarbst G_{q,t}$ will follow from the continuity of $G_{q,t}$. Unfortunately, we cannot
apply Proposition \ref{prop:basic estimate} to either $\dbarb G_{q,t}\vp$ or $\dbarbst G_{q,t}\vp$ because neither are $(0,q)$-forms. Instead, we estimate
directly:
\begin{align*}
\|\dbarb G_{q,t}\vp\|_s^2 &+ \|\dbarbst G_{q,t}\vp\|_s^2
\leq C_t ( \norm \Lambda^s \dbarb G_{q,t}\vp\norm_t^2 + \norm \Lambda^s \dbarbst G_{q,t}\vp\norm_t^2 ) \\
&= C_t\Big( \la \Lambda^s \vp, \Lambda^s G_{q,t}\vp \ra_t + \la  \Lambda^s \dbarb G_{q,t}\vp,  [\Lambda^s, \dbarb] G_{q,t}\vp \ra_t
+ \la [\dbarbst, \Lambda^s] \dbarb G_{q,t}\vp, \Lambda^s G_{q,t}\vp \ra_t \\
 &+ \la  \Lambda^s \dbarbst G_{q,t}\vp,  [\Lambda^s, \dbarbst] G_{q,t}\vp \ra_t
+ \la [\dbarb, \Lambda^s] \dbarbst G_{q,t}\vp, \Lambda^s G_{q,t}\vp \ra_t \Big) \\
&\leq C_{t,s} (\| \vp \|_s^2 + \| G_{q,t}\vp \|_s^2) \leq C_{t,s} \|\vp\|_s^2.
\end{align*}

%
%
\subsection{The Szeg\"o projection $S_{q,t}$}
The Szeg\"o projection $S_{q,t}$
is the projection of $L^2_{0,q}(M)$ onto $\ker\dbarb$. We claim that
\[
S_{q,t} = I - \dbarbst\dbarb G_{q,t} = I - G_{q,t} \dbarbst\dbarb.
\]
The second equality follows from (\ref{eqn:commuting dbarb dbarbst by G_q}). Observe that if $\vp\in\Null(\dbarb)$, then
$( I - G_{q,t} \dbarbst\dbarb)\vp = \vp$, as desired. If $\vp\perp \Null(\dbarb)$, then $\vp \perp \H^q_t$, so
$\vp = \dbarbst\dbarb G_{q,t}\vp + \dbarb\dbarbst G_{q,t}\vp$. We claim that $\vp = \dbarbst\dbarb G_{q,t}\vp$. Let
$u = \dbarbst\dbarb G_{q,t}\vp$. Then $u$ is the canonical solution to $\dbarb u = \dbarb\vp$, so $\dbarb(\vp-u)=0$. However,
$\vp \perp\Null(\dbarb)$, so $u =\vp$, and $0 = \vp-u = (I - \dbarbst\dbarb G_{q,t})\vp$, as desired.

\begin{prop}\label{prop:the Szego kernel is continuous on H^s}
Let $M$ be as in Theorem \ref{thm:main theorem for weighted spaces}. If $t\geq T_s$, then the Szeg\"o kernel
$S_{q,t}$ is continuous on  $H^s_{0,q}(M)$.
\end{prop}

\begin{proof}This argument uses ideas from \cite{BoSt90}.
Given $\vp\in L^2_{0,q}(M)$, we know that $\dbarbst\dbarb G_{q,t} \vp\in L^2_{0,q}(M)$, but we have no quantitative bound. However,
\[
\norm \dbarbst\dbarb G_{q,t} \vp \norm_t^2 =
\la \dbarb\dbarbst\dbarb G_{q,t}\vp , \dbarb G_{q,t} \vp \ra_t = \la \dbarb \vp, \dbarb G_{q,t} \vp \ra_t
\leq \norm \vp \norm_t \norm \dbarbst\dbarb G_{q,t} \vp \norm_t.
\]
This proves continuity in $L^2_{0,q}(M)$.

Now let $s>0$. It suffices to show
\begin{equation}\label{eqn:good enough bound for Szego continuity in H^s}
\norm \Lambda^s \dbarbst\dbarb G_{q,t}\vp \norm_t^2 \leq C_{s,t}\norm \Lambda^s \vp\norm_t^2.
\end{equation}
We cannot simply integrate by parts as in the $L^2$-case because we do not know if $\Lambda^s \dbarbst\dbarb S_{q,t}\vp$ is finite. As above, we can
avoid this issue by an elliptic regularity argument. Using the operators $\Gqdt$ from \S6.4, we have (if $\delta$ is small enough)
\begin{align*}
\norm \Lambda^s \dbarbst\dbarb\Gqdt \vp \norm_t^2
= &\la \Lambda^s \dbarb\dbarbst\dbarb \Gqdt \vp, \Lambda^s \dbarb\Gqdt\vp \ra_t
+ \la [\dbarb,\Lambda^s]\dbarbst\dbarb\Gqdt\vp, \Lambda^s\dbarb \Gqdt \vp \ra_t \\
&+ \la \Lambda^s \dbarbst\dbarb\Gqdt \vp, [\Lambda^s,\dbarbst] \dbarb\Gqdt\vp \ra_t \\
\leq& C_{s,t}( \norm \Lambda^s \vp \norm_t + \norm\Lambda^s \dbarb \Gqdt \vp \norm_t) \norm\Lambda^s \dbarbst\dbarb\Gqdt \vp \norm_t.
\end{align*}
Using that the continuity of $\dbarb \Gqdt$ in $H^s_{0,q}(M)$ is uniform in $\delta$ (for small $\delta$), we have
\begin{equation}\label{eqn:Szego regularity except for delta}
\norm \Lambda^s \dbarbst\dbarb\Gqdt \vp \norm_t
\leq C_{s,t}( \norm \Lambda^s \vp \norm_t + \norm\Lambda^s \dbarb G_{q,t} \vp \norm_t)
\leq C_{s,t}  \norm \Lambda^s \vp \norm_t.
\end{equation}
As earlier, we can take an appropriate limit as $\delta\to 0$ to establish the bound in \eqref{eqn:Szego regularity except for delta} with $\delta=0$.
\end{proof}

%
%
\subsection{Results for levels $(0,q-1)$ and $(0,q+1)$.}
We now show continuity of the canonical solution operators $G_{q,t}\dbarbst: H^s_{0,q+1}(M)\to H^s_{0,q}(M)$ and
$G_{q,t}\dbarb :H^s_{0,q-1}(M)\to H^s_{0,q}(M)$, and the Szeg\"o projection $S_{q-1,t} = I - \dbarbst G_{q,t}\dbarb: H^s_{0,q-1}(M) \to H^s_{0,q-1}(M)$. We cannot
express the Szeg\"o kernel of $(0,q+1)$-forms in terms of $G_{q,t}$
because the only candidate  is
$\dbarb G_{q,t}\dbarbst$, but this object annihilates $t$-harmonic forms (which ought to remain unchanged by $S_{q+1,t}$). Since $H^{s+1}_{0,q-1}(M)$ is dense in $H^s_{0,q-1}(M)$ and $G_{q,t}$ preserves $H^s_{0,q}(M)$, we may assume that $\vp\in H^{s+1}_{0,q-1}(M)$. Then
\begin{align*}
\norm \Lambda^s G_{q,t}\dbarb\vp \norm_t^2
&= \la \hspace{-.375cm}\underbrace{\dbarbst G_{q,t}}_{\text{bounded in } H^s} \hspace{-.4cm}\Lambda^s G_{q,t}\dbarb \vp, \Lambda^s\vp \ra_t
+ \la \Lambda^s G_{q,t} \dbarb\vp, [\Lambda^s, G_{q,t}\dbarb] \vp \ra_t \\
&\leq C_{s} \norm \Lambda^s G_{q,t} \dbarb\vp \norm_t \norm \Lambda^s \vp \norm_t.
\end{align*}
The right hand side is finite since $\dbarb\vp \in H^s_{0,q}(M)$ by assumption. Thus, $G_{q,t}\dbarb:H^s_{0,q-1}(M)\to H^s_{0,q}(M)$ is bounded.
A similar argument shows that  $G_{q,t}\dbarbst: H^s_{0,q+1}(M)\to H^s_{0,q}(M)$ is continuous.

For the Szeg\"o projection, we investigate the boundedness of
\begin{multline*}
\norm \Lambda^s \dbarbst G_{q,t} \dbarb \vp \norm_t^2
= \la \Lambda^s \dbarb\dbarbst G_{q,t} \dbarb\vp, \Lambda^s G_{q,t}\dbarb\vp \ra_t
+ \la \Lambda^s \dbarbst G_{q,t} \dbarb\vp, [\Lambda^s,\dbarbst]G_{q,t}\dbarb\vp \ra_t \\
+ \la[\dbarbst,\Lambda^s]\dbarbst G_{q,t}\dbarb\vp,\Lambda^s G_{q,t}\dbarb\vp \ra_t.
\end{multline*}
Since $\dbarb\vp$ is $\dbarb$-closed, $ \dbarb\dbarbst G_{q,t} \dbarb\vp = \dbarb\vp$, so
\begin{multline*}
 \la \Lambda^s \dbarb\dbarbst G_{q,t} \dbarb\vp, \Lambda^s G_{q,t}\dbarb\vp \ra_t
 =\\ \la \Lambda^s\vp, \Lambda^s \dbarbst G_{q,t} \dbarb\vp \ra_t +
 \la [\Lambda^s,\dbarb] \vp, \Lambda^s G_{q,t}\dbarb \vp \ra_t + \la \Lambda^s\vp, [\Lambda^s,\dbarbst] G_{q,t}\dbarb\vp \ra_t \\
\leq C_s (\norm\Lambda^s\vp\norm_t \norm \Lambda^s \dbarbst G_{q,t} \dbarb\vp\norm_t + \norm\Lambda^s\vp\norm_t ^2).
\end{multline*}
Thus, we have
\begin{align*}
\norm \Lambda^s \dbarbst G_{q,t} \dbarb \vp \norm_t^2
\leq C_{s,t} (\norm\Lambda^s\vp\norm_t \norm \Lambda^s \dbarbst G_{q,t} \dbarb\vp\norm_t + \norm\Lambda^s\vp\norm_t ^2).
\end{align*}
Using a small constant/large constant argument and absorbing terms, we have the continuity of the Szeg\"o projection
in $H^s_{0,q-1}(M)$.

The continuity of the solution operator $\dbarbst G_{q,t}$ immediately gives closed range of  $\dbarb$ from $H^s_{0,q-1}(M)$ to $H^s_{0,q}(M)$.
Similarly, the boundedness of the operator $\dbarb G_{q,t}$ immediately gives closed range of $\dbarbs$ from $H^s_{0,q+1}(M)$ to $H^s_{0,q}(M)$.

%
%
\subsection{Exact and global regularity for $\dbarb$.}
In this section, we prove that if $\alpha \in C^\infty_{0,\tilde q+1}(M)$ satisfies $\dbarb\alpha=0$ and $\alpha \perp \H_t^{\tilde q}$,
then there exists $u\in C^\infty_{0,\tilde q}(M)$ so that $\dbarb u = \alpha$ where
$\tilde q = q$ or $q-1$.
We follow the argument
in \cite{Nic06}, Lemma 5.10. We start by showing that if $k$ is fixed and $s>k$, then $H^s_{0,\tilde q}(M) \cap \Null(\dbarb)$ is dense in
$H^k_{0,\tilde q}(M)\cap\Null(\dbarb)$. Let $g\in H^k_{0,\tilde q}(M)\cap \Null(\dbarb)$. Since $C^\infty_{0,\tilde q}(M)$ is dense in $H^k_{0,\tilde q}(M)$,
there exists
a sequence $g_j\in C^\infty_{0,\tilde q}(M)$ so that $g_j\to g$ in $H^k_{0,\tilde q}(M)$. Let $t\geq T_s$ and set
$\tilde g_j = S_{\tilde q,t} g_j$. By the continuity of $S_{\tilde q,t}$ in $H^s_{0,\tilde q}(M)$, $\tilde g_j \in H^s_{0,\tilde q}(M)$. Moreover,
since $g = S_{\tilde q,t} g$, it follows that
\[
\lim_{j\to\infty} \| \tilde g_j - g \|_k^2 = \lim_{j\to\infty}\| S_{\tilde q,t}(g_j-g) \|_k^2 \leq C_{k,t}\lim_{j\to\infty} \| g_j - g \|_k^2 =0.
\]

Next, since $\alpha = \dbarb\dbarbst G_{\tilde q,t}\alpha$ or $\dbarb G_{\tilde q,t}\dbarbst \alpha$
for all sufficiently large $t$, by choosing an appropriate sequence $t_k\to \infty$, there exists
$u_k = \dbarbst G_{\tilde q,t_k}\alpha$ or $G_{\tilde q,t_k}\dbarbst \alpha \in H^k_{0,\tilde q}(M)$
so that $\dbarb u_k = \alpha$. We will construct a sequence $\tilde u_k$ inductively. Let
$\tilde u_1 = u_1$. Assume that $\tilde u_k$ has been defined so that $\tilde u_k \in H^k_{0,\tilde q}(M)$, $\dbarb \tilde u_k = \alpha$, and
$\| \tilde u_k - \tilde u_{k-1} \|_{k-1} \leq 2^{k-1}$. We will now construct $\tilde u_{k+1}$. Note that $\dbarb(u_{k+1} - \tilde u_k)=0$. By the density
argument above, there exists $v_{k+1}\in H^{k+1}_{0,\tilde q}(M)\cap \Null(\dbarb)$ so that if $\tilde u_{k+1} = u_{k+1} + v_{k+1}$, then
$\| \tilde u_{k+1} - \tilde u_k \|_k \leq 2^{-k}$. Finally, set
\[
u = \tilde u_1 + \sum_{k=1}^\infty (\tilde u_{k+1}-\tilde u_k) = \tilde u_j + \sum_{k=j}^\infty (\tilde u_{k+1}-\tilde u_k), \quad j\in\N.
\]
The sum telescopes and it is clear that $u\in H^j_{0,\tilde q}(M)$ for all $j\in \N$ and $\dbarb u = \alpha$. Thus, $u\in C^\infty_{0,\tilde q}(M)$.

%
%
\section{Proof of Theorem \ref{thm:main theorem for unweighted}}\label{sec:main theorem unweighted}
From (\ref{eqn:norm equivalence}), we know that weighted $L^2(M)$ and $L^2(M)$ are equivalent spaces. Thus, from
Theorem \ref{thm:main theorem for weighted spaces}, we know that $\dbarb:L^2_{0,q-1}(M)\to L^2_{0,q}(M)$ and
$\dbarb:L^2_{0,q}(M)\to L^2_{0,q+1}(M)$ have closed range. Again by H\"ormander, Theorem 1.1.1, this proves that
$\dbarbs:L^2_{0,q}(M)\to L^2_{0,q-1}(M)$ and $\dbarbs:L^2_{0,q+1}(M)\to L^2_{0,q}(M)$ have closed range. Consequently,
the Kohn Laplacian $\Box_b = \dbarb\dbarbs + \dbarbs\dbarb$ has closed range on $ L^2_{0,q}(M)$ and the remainder of
the theorem follows by standard arguments.
This concludes the proof of Theorem \ref {thm:main theorem for unweighted}.

\begin{rem}\label{rem:weighted vs. unweighted bound 2} This is more quantitative discussion of
Remark \ref{rem:weighted vs. unweighted bound}. In particular, from the proof of Theorem \ref{thm:main theorem for unweighted}, we have
the closed range bound for appropriate $(0,q)$-forms $\vp$ (using \eqref{eqn:norm equivalence}),
\[
\|\vp\|_0^2 \leq \frac {1}{c_t} \| \vp \|_{t}^2 \leq \frac C{c_t} \| \dbarb \vp \|_t^2 \leq \frac{C C_t}{c_t} \|\dbarb \vp\|_0^2.
\]
Thus, the closed range constants for $\dbarb$, $\dbarbs$, and $\Box_b$ in unweighted $L^2(M)$ depend on the size of $\lp$ and $\lm$.
\end{rem}

\bibliographystyle{alpha}
\bibliography{mybib}

\end{document}